\documentclass[a4paper,10pt,reqno]{amsart}
\usepackage{amssymb,latexsym,bbm,amsmath,amsfonts}
\usepackage{amsthm}
\usepackage[mathscr]{eucal}
\usepackage{rotating}
\usepackage{multirow}
\usepackage{slashbox}

\numberwithin{equation}{section}
\newtheorem{theorem}{Theorem}
\newtheorem{lemma}{Lemma}
\newtheorem{proposition}{Proposition}

\theoremstyle{definition}

\newtheorem{example}[theorem]{Example}

\theoremstyle{remark}

\begin{document}

\title[Asymptotic Classification of Affine SDEs]
{Classification of the Asymptotic Behaviour of Globally Stable Linear Differential Equations with Respect to 
State--independent Stochastic Perturbations}

\author{John A. D. Appleby}
\address{Edgeworth Centre for Financial Mathematics, School of Mathematical
Sciences, Dublin City University, Glasnevin, Dublin 9, Ireland}
\email{john.appleby@dcu.ie} \urladdr{webpages.dcu.ie/\textasciitilde
applebyj}

\author{Jian Cheng}
\address{Edgeworth Centre for Financial Mathematics, School of Mathematical
Sciences, Dublin City University, Glasnevin, Dublin 9, Ireland}
\email{jian.cheng2@mail.dcu.ie}

\author{Alexandra Rodkina}
\address{The University of the West Indies, Mona Campus
Department of Mathematics and Computer Science Mona, Kingston 7,
Jamaica} \email{alexandra.rodkina@uwimona.edu.jm}

\thanks{The first two authors gratefully acknowledges Science Foundation
Ireland for the support of this research under the Mathematics
Initiative 2007 grant 07/MI/008 ``Edgeworth Centre for Financial
Mathematics''.} 
\date{22 October 2012}

\begin{abstract}
In this paper we consider the global stability of solutions of an
affine stochastic differential equation. The differential equation
is a perturbed version of a globally stable linear autonomous
equation with unique zero equilibrium where the diffusion
coefficient is independent of the state. We find necessary and
sufficient conditions on the rate of decay of the noise intensity
for the solution of the equation to be globally asymptotically
stable, stable but not asymptotically stable, and unstable, each
with probability one. In the case of stable or bounded solutions, or when 
solutions are a.s. unstable asymptotically stable in mean square, it follows that the 
norm of the solution has zero liminf, by virtue of the fact that $\|X\|^2$ has zero pathwise average a.s.s 
Sufficient conditions guaranteeing the different types of 
asymptotic behaviour which are more readily checked are developed. 
It is also shown that noise cannot stabilise solutions, and that the results can be 
extended in all regards to affine stochastic differential equations with periodic coefficients.  
\end{abstract}

\maketitle

\section{Introduction}


In this paper we analyse the asymptotic behaviour of finite--dimensional
affine stochastic differential equations. We suppose that in the absence of a stochastic perturbation
that there is unique and globally stable equilibrium at zero. The perturbation can be viewed as an \emph{external}
force, in the sense that the intensity of the entries in the diffusion matrix are independent of the state.

Therefore we may consider the underlying$d$--dimensional ordinary (deterministic) differential equation
\[
x'(t)=Ax(t), \quad t\geq 0; \quad x(0)=\xi\in\mathbb{R}^d.
\]
Here we have that $A$ is a $d\times d$ real matrix.
 Since we are
presuming that there is a unique equilibrium at zero, and that it is
globally stable, we assume that all the eigenvalues of $A$ have
negative real parts. One of the important tasks in this paper is
to classify the asymptotic behaviour of the stochastic differential
equation
\begin{equation} \label{eq.introgensde}
dX(t)=AX(t)\,dt + \sigma(t)\,dB(t)
\end{equation}
In this setting, $\sigma$ is a continuous and deterministic function and $B$ is
a finite dimensional Brownian motion. Specifically, we let
\begin{equation}\label{eq.sigmacns}
\sigma\in C([0,\infty);\mathbb{R}^{d\times r})
\end{equation}
and $B$ be an $r$--dimensional standard Brownian motion.

Since equations with state--independent noise should be in general
simpler to analyse that state--dependent case, and their
applications are of interest, it is not surprising that such
equations have attracted a lot of attention. Liapunov function
techniques have been applied to study their asymptotic stability in
Khas'minski~\cite{Has}, with a lot of emphasis given to equations
with perturbations $\sigma$ being in $L^2(0,\infty)$. However, in a
pair of papers in 1989, Chan and Williams~\cite{ChanWill:1989} and
Chan\cite{Chan:1989} demonstrated that the stability of global
equilibria in these systems could be preserved with a much slower
rate of decay in $\sigma$: in fact, they showed that provided the
noise perturbation decayed monotonically in its intensity, then
solutions converged to the equilibrium with probability one if and
only if
\[
\lim_{t\to\infty} \|\sigma(t)\|^2\log t =0.
\]
These results also required strong assumptions on the strength of
the nonlinear feedback. Shortly thereafter,
Rajeev~\cite{Rajeev:1993} demonstrated that these results could be
generalised to equations with some non--autonomous features, and
some results on bounded solutions were obtained. In parallel, Mao
demonstrated in~\cite{Mao:1992} that a polynomial rate of decay of
solutions was possible if the perturbation intensity decayed at a
polynomial rate. These results were extended to neutral functional
differential equations by Mao and Liao in~\cite{MaoLiao:1996}, with
exponential decaying upper bounds on the intensity giving rise to an
exponential convergence rate in the solution.

After this, Appleby and his co--authors extended Chan and Williams'
results to stochastic functional differential equations
\cite{AppRod:2005SDDE} and to Volterra equations especially (see
Appleby and Appleby and Riedle \cite{Appleby:2002,
ApplebyRiedle:2006}), with extensions to discrete Volterra equations
appearing in Appleby, Riedle and Rodkina~\cite{JAARMR:2009}.
Necessary and sufficient conditions for exponential stability in
linear Volterra equation in the presence of fading noise was studied
in~\cite{ApplebyFreeman:2003}.

One of the papers which has most influence on this work is Appleby,
Gleeson and Rodkina~\cite{JAJGAR:2009}, which returns directly to
the nonlinear equations studied by Chan and Williams in
\cite{ChanWill:1989}. In it, the monotonicity assumptions on
$\sigma$ were completely relaxed, and the mean reversion strength
was also considerably weakened. Moreover, results on unbounded and
unstable solutions also appeared for the first time. However, the
finite dimensional case was not addressed, nor was a complete
classification of the dynamics presented. The goal this paper is to
address this of the thesis is to address each of these shortcomings

An important idea which appears in
\cite{AppRod:2005SDDE,ApplebyRiedle:2006,JAJGAR:2009} in various
forms is that many facts about more complicated stochastic
differential, functional or Volterra equations with
state--independent noise can be inferred from a much simpler
$d$--dimensional equation whose solution $Y$ which is given by
\begin{equation}\label{def.Y}
dY(t)=-Y(t)dt+\sigma(t)dB(t),\quad t\geq 0;\quad Y(0)=0.
\end{equation}
In fact, we demonstrate that $X$ and $Y$ have equivalent asymptotic
behaviour, in the sense that $X$ converges to zero if and only if
$Y$ does; is bounded but not convergent if and only if $Y$ is; and
is unbounded if and only if $Y$ is.

Therefore, the question of analysing the asymptotic behaviour of the
general linear equation reduces to that of studying the special
linear equation \eqref{def.Y}. If $\sigma$ is identically zero, it
follows that the solution of
\[
y'(t)=-y(t),\quad t\geq 0; \quad y(0)=0.
\]
obeys $y(t)=0$ for all $t\geq 0$ if $y(0)=0$. The question naturally
arises as under what condition on $\sigma$ does the solution $Y(t)$
obey
\begin{equation}\label{eq.Yto0}
\lim_{t\to \infty}Y(t)=0, \quad a.s.
\end{equation}
It is shown in ~\cite{ChanWill:1989} that $Y(t)$ obeys
\eqref{eq.Yto0} in the one--dimensional case if
\[
\lim_{t\to \infty}\sigma^2(t)\log t=0.
\]
Moreover in ~\cite{ChanWill:1989}, it is shown that if
$t\to\sigma^2(t)$ is decreasing to zero, and $Y(t)$ obeys
\eqref{eq.Yto0}, then we must have $\lim_{t\to
\infty}\sigma^2(t)\log t=0$. These results are extended to
finite--dimensions in ~\cite{Chan:1989}. In~\cite{JAJGAR:2009},
monotonicity assumptions on $\sigma$ are relaxed, and results for
unbounded solutions for \eqref{def.Y} are presented. However, none
of these papers classify all the possible types of asymptotic
behaviour of $Y$. This situation was rectified in the scalar case
$(d=1)$ in ~\cite{JAJCAR:2011dresden}, in which the asymptotic
behaviour of solutions of \eqref{def.Y} are classified.

In this paper, we extend the classification of solutions to the
general finite--dimensional case. In fact, we characterise the
convergence, boundedness and unboundedness of solutions of
\eqref{def.Y}, and this leads in turn to a classification of the
convergence, boundedness and unboundedness of solutions of
\eqref{eq.introgensde}. Moreover, it turns out that neither
pointwise convergence rates nor pointwise monotonicity are needed in
order to achieve this classification. Our main results show that $X$
obeys $\lim_{t\to\infty} X(t)=0$ a.s. if and only if
\begin{equation}\label{eq.sum}
S_h(\epsilon)=\sum_{n=0}^{\infty}\left\{1-\Phi\left(\frac{\epsilon}{\sqrt{\int_{nh}^{(n+1)h}\|\sigma(s)\|^2_{F}ds}}
\right) \right\}<+\infty,\quad \text{for every}\quad \epsilon>0,
\end{equation}
where $\Phi$ is the distribution function of a standardised normal
random variable and $h$ is any positive constant. We also show that
in contrast to \eqref{eq.sum}, if $S_h(\epsilon)$ is infinite for
all $\epsilon$, we have that $\limsup_{t\to
\infty}\|X(t)\|=+\infty$; while if the sum is finite for some
$\epsilon$ and infinite for others, then $c_1\leq \limsup_{t\to
\infty}\|X(t)\|\leq c_2$ a.s., where $0<c_1\leq c_2<+\infty$ are
deterministic and $\liminf_{t\to\infty} \|X(t)\|=0$ a.s. In this
last case, when $X$ is bounded, the solution spends most of the time
close to zero, because
\begin{equation} \label{eq.aveX0}
\lim_{t\to\infty} \frac{1}{t}\int_0^t \|X(s)\|^2\,ds=0, \quad
\text{a.s.}
\end{equation}
Since $S_h(\epsilon)$ is monotone in $\epsilon$, it can be seen that
we can describe the asymptotic behaviour for every function
$\sigma$, and that, moreover, the stability, boundedness or
unboundedness of the solution depends on $\sigma$ only through the
overall intensity of the perturbation through the Frobenius norm
$\|\sigma\|_F$, and not through the configuration of the
perturbation and its interaction with the matrix $A$. Moreover, it
can be seen that these conditions which guarantee convergence,
boundedness or unboundedness are independent of the matrix $A$.
Also, by virtue of the form of $S_h(\epsilon)$ and the equivalence
of all norms on $\mathbb{R}^{d\times r}$, it follows that the
asymptotic behaviour relies only on $\|\sigma\|$, where $\|\cdot\|$
is any norm in $\mathbb{R}^{d\times r}$.

Since the underlying deterministic differential equations is assumed
to be stable, it is of interest to determine its response to fading
noise perturbations. In this case, we can find a quite general
characterisation of ``fading noise'' which yields a more
comprehensive picture about the asymptotic behaviour of $X$. If the
fading noise condition is $\int_{nh}^{(n+1)h}
\|\sigma(s)\|^2_F\,ds\to 0$ as $n\to\infty$---which is automatically
true in the case that $X$ is bounded or stable--- is assumed in the
case when $S_h(\epsilon)=+\infty$, then the process $\|X\|$ is
recurrent on $(0,\infty)$, because  $\liminf_{t\to\infty}
\|X(t)\|=0$ and $\limsup_{t\to \infty}\|X(t)\|=+\infty$ a.s.
Furthermore $X$ spends most of the time close to zero in the sense
that \eqref{eq.aveX0} holds. Hence, under the fading noise
condition, we can see that we always have $\liminf_{t\to\infty}
\|X(t)\|=0$ and \eqref{eq.aveX0} holding, regardless of the
finiteness of $S_h$ but that $\limsup_{t\to\infty} \|X(t)\|$ is
zero, positive and finite, or infinite a.s., according as to whether
$S_h$ is always finite, sometimes finite, or always infinite. It is
worth remarking that the fading noise condition we choose is
precisely that which is necessary and sufficient for the mean square
stability of solutions of \eqref{eq.introgensde}.

Given that we are dealing with a continuous time equation, it seems
appropriate that the conditions which enable us to characterise the
asymptotic behaviour should be ``continuous'' rather than
``discrete''. The finiteness condition on $S_h(\epsilon)$, which
relies on a particular partition of time, and the convergence of a
sum, can certainly be seen as a ``discrete'' condition, in this
sense. Therefore, we develop an integral condition on $\sigma$ which
is equivalent to the summation condition in \eqref{eq.sum}. More
precisely, we define
\begin{equation} \label{eq.Ieps}
I_c(\epsilon)=\int_0^\infty \sqrt{\int_{t}^{t+c}
\|\sigma(s)\|^2_F\,ds} \exp\left(-\frac{\epsilon^2/2}{\int_{t}^{t+c}
\|\sigma(s)\|^2_F}\right)\chi_{(0,\infty)}\left(\int_{t}^{t+c}
\|\sigma(s)\|^2_F\right)\,ds
\end{equation}
for arbitrary $c>0$. We then show that $I_c(\epsilon)$ being finite
for all $\epsilon$ implies that $X$ tends to $0$; if $I_c(\epsilon)$
is infinite for all $\epsilon$ then $X$ is unbounded; and if
$I_c(\epsilon)$ is finite for some $\epsilon$ and infinite for
others, then $X$ is bounded but not convergent to zero. The value of
$c$ turns out to be unimportant, and can be chosen to be unity for
convenience. As might be guessed, the finiteness of $I_c(\epsilon)$
for all $\epsilon$ is equivalent to the finiteness of
$S_h(\epsilon)$ for all $\epsilon$; $I_c(\epsilon)$ being infinite
for all $\epsilon$ is equivalent to $S_h(\epsilon)$ being infinite
for all $\epsilon$; and $I_c(\epsilon)$ is finite for some
$\epsilon$ and infinite for others if and only if $S_h(\epsilon)$
is.

Although \eqref{eq.sum} or $I_c(\epsilon)$ being finite are
necessary and sufficient for $X$ to obey $\lim_{t\to\infty} X(t)=0$
a.s., these conditions may be hard to apply in practice. For this
reason we also deduce sharp sufficient conditions on $\sigma$ which
enable us to determine for which value of $\epsilon$ the functions
$S_h(\epsilon)$ or $I_c(\epsilon)$ are finite. One such condition is
the following: if it is known for some $c>0$ that
\[
\lim_{t\to\infty} \int_t^{t+c} \|\sigma(s)\|^2_F\,ds \log t = L\in
[0,\infty],
\]
then $L=0$ implies that $X$ tends to zero a.s.; $L$ being positive
and finite implies $X$ is bounded, but does not converge to zero;
and $L$ being infinite implies $X$ is unbounded. In the case when
$t\mapsto \|\sigma(t)\|^2=:\Sigma_1(t)^2$ or $t\mapsto\int_t^{t+1}
\|\sigma(s)\|^2\,ds=:\Sigma_2(t)^2$ are nonincreasing functions, it
can also be seen that $X(t)\to 0$ as $t\to\infty$ a.s. is equivalent
to $\Sigma_i(t)^2\log t=0$.

One other result of note is established. We ask: is it possible for
solutions of the unperturbed ODE $x'(t)=Ax(t)$ to be unstable, but
solutions of the SDE to be stable for some nontrivial $\sigma$? In
other words, can the noise \emph{stabilise} solutions? We prove that
it cannot, in the sense that if there are a representative and
finite collection of initial conditions $\xi$ for which $X(t,\xi)$
tends to zero with positive probability, then it must be the case
that all the eigenvalues of $A$ have negative real parts, and that
$S(\epsilon)$ is finite for all $\epsilon>0$. These conditions are
therefore equivalent to $\lim_{t\to\infty}X(t,\xi)=0$ a.s. for each
initial condition $\xi$.

The results on the equation \eqref{def.Y} are of more general
utility than in the linear autonomous case. We give an example here
of how they can be used to classify the asymptotic behaviour of a
periodic linear ODE. We plan to show in other works that the
asymptotic behaviour of $Y$ can be used in both the scalar and
finite--dimensional case to understand the asymptotic behaviour of
the general nonlinear SDE
\[
dX(t)=-f(X(t))\,dt + \sigma(t)\,dB(t)
\]
which, in the absence of a stochastic perturbation, has a unique
globally asymptotically stable equilibrium at zero.

The next section states and discusses the main results, with proofs
and supporting lemmata in the following section. Then we discuss
the sufficient conditions on $\sigma$ for stability with proofs and
supporting lemmata.

\section{Discussion and Statement of Main Results}
\subsection{Notation}
In advance of stating and discussing our main results, we introduce
some standard notation. Let $d$ and $r$ be integers. We denote by
$\mathbb{R}^d$ $d$--dimensional real--space, and by
$\mathbb{R}^{d\times r}$ the space of $d\times r$ matrices with real
entries. Here $\mathbb{R}$ denotes the set of real numbers. We
denote the maximum of the real numbers $x$ and $y$ by $x\vee y$ and
the minimum of $x$ and $y$ by $x\wedge y$. If $x$ and $y$ are in
$\mathbb{R}^d$, the standard innerproduct of $x$ and $y$ is denoted
by $\langle x,y\rangle$. The standard Euclidean norm on
$\mathbb{R}^d$ induced by this innerproduct is denoted by
$\|\cdot\|$. If $A\in \mathbb{R}^{d\times r}$, we denote the entry
in the $i$--th row and $j$--th column by $A_{ij}$. For $A\in
\mathbb{R}^{d\times r}$ we denote the Frobenius norm of $A$ by
\[
\|A\|_F=\left( \sum_{j=1}^r \sum_{i=1}^d \|A_{ij}\|^2\right)^{1/2}.
\]
Let $C(I;J)$ denote the space of continuous functions $f:I\to J$
where $I$ is an interval contained in $\mathbb{R}$ and $J$ is a
finite dimensional Banach space. We denote by
$L^2([0,\infty);\mathbb{R}^{d\times r})$ the space of Lebesgue
square integrable functions $f:[0,\infty)\to\mathbb{R}^{d\times r}$
such that $\int_0^\infty \|f(s)\|_F^2\,ds < + \infty$.

\subsection{Main results}
Our first result demonstrates that it is necessary to classify completely the asymptotic behaviour of
only a \emph{single} affine stochastic differential equation in order to classify the asymptotic behaviour for \emph{all}
affine stochastic differential equations with the same diffusion coefficient, for which the underlying deterministic linear differential
equation is asymptotically stable.

To make this precise, let $d$ be an integer and $A$ be a $d\times d$ matrix with real
entries, and consider the deterministic linear differential equation
\begin{equation} \label{eq.linode}
x'(t)=Ax(t), \quad t\geq 0; \quad x(0)=\xi\in \mathbb{R}^d,
\end{equation}
and also consider the stochastically perturbed version of \eqref{eq.linode}, namely
\begin{equation} \label{eq.stochlinode}
dX(t)=AX(t)\,dt + \sigma(t)\,dB(t), \quad t\geq 0; \quad X(0)=\xi\in
\mathbb{R}^d.
\end{equation}
Our first main result states that if $Y$ has certain types of almost sure asymptotic behaviour, then $X$ inherits that almost sure asymptotic behaviour.
\begin{theorem} \label{theorem.XYasy}
Let $A$ be a $d\times d$ real matrix for which all eigenvalues have negative real
parts. Let $\sigma$ obeys \eqref{eq.sigmacns}, $Y$ be the
unique continuous adapted process which obeys \eqref{def.Y}, and $X$ be the
unique continuous adapted process which obeys \eqref{eq.stochlinode}. Then
\begin{itemize}
\item[(A)] If $\lim_{t\to\infty} Y(t)=0$ a.s., then $\lim_{t\to\infty} X(t)=0$, a.s.
\item[(B)] If there exist $0\leq c_1\leq c_2<+\infty$ such that
\[
c_1\leq \liminf_{t\to\infty} \|Y(t)\|\leq \limsup_{t\to\infty} \|Y(t)\|\leq c_2, \quad\text{a.s.}
\]
then there exist $0\leq c_3\leq c_4<+\infty$ such that
\[
c_3\leq \liminf_{t\to\infty} \|X(t)\|\leq \limsup_{t\to\infty} \|X(t)\|\leq c_4, \quad\text{a.s.}
\]
\item[(C)] If $\limsup_{t\to\infty} \|Y(t)\|=+\infty$ a.s., then $\limsup_{t\to\infty} \|X(t)\|=+\infty$ a.s.
\end{itemize}
\end{theorem}

Therefore, the asymptotic behaviour of $X$ can be classified, provided the hypothesised asymptotic behaviour of $Y$ in Theorem~\ref{theorem.XYasy}
can be established. Our next result claims that such a classification can be achieved. Before it can be stated, we make some observations and fix notation. First, we see that $Y$ has the representation
\begin{equation}   \label{eq.Yformmult}
Y(t)=e^{-t}\int_0^t e^s \sigma(s)\,dB(s), \quad t\geq 0.
\end{equation}
Denote by $\Phi:\mathbb{R}\to \mathbb{R}$ the distribution function of a
standard normal random variable
\begin{equation} \label{def.Phi}
\Phi(x)=\frac{1}{\sqrt{2\pi}}\int_{-\infty}^x e^{-u^2/2}\,du, \quad
x\in\mathbb{R}.
\end{equation}
We interpret $\Phi(-\infty)=0$ and $\Phi(\infty)=1$. 
Define $S_h$ by
\begin{equation} \label{def.Sh}
S_h(\epsilon)=\sum_{n=0}^\infty
\left\{1-\Phi\left(\frac{\epsilon}{\sqrt{\int_{nh}^{(n+1)h}
\|\sigma(s)\|^2_F\,ds}} \right) \right\}.
\end{equation}
Since $S_h$ is a monotone function of $\epsilon$, it is the case
that either (i) $S_h(\epsilon)$ is finite for all $\epsilon>0$; (ii)
there is $\epsilon'>0$ such that for all $\epsilon>\epsilon'$ we
have $S_h(\epsilon)<+\infty$ and $S_h(\epsilon)=+\infty$ for all
$\epsilon<\epsilon'$; and (iii) $S_h(\epsilon)=+\infty$ for all
$\epsilon>0$. The finiteness of the sum $S_h(\epsilon)$ may be hard to estimate because $\Phi$ is not known in closed form. However, the asymptotic behaviour of $1-\Phi$ is well--known via Mill's estimate cf., e.g., \cite[Problem 2.9.22]{K&S}
\begin{equation} \label{eq.millsasy}
\lim_{x\to\infty}\frac{1-\Phi(x)}{x^{-1}e^{-x^2/2}}=\frac{1}{\sqrt{2\pi}},
\end{equation}
so it is possible to determine whether $S_h(\epsilon)$ is finite according as to whether
\begin{equation} \label{def.Shpr}
S_h'(\epsilon)=\sum_{n=1}^\infty \sqrt{\int_{nh}^{(n+1)h} \|\sigma(s)\|^2_F\,ds} \cdot
\exp\left(-\frac{\epsilon^2}{2\int_{nh}^{(n+1)h} \|\sigma(s)\|^2_F\,ds}\right),
\end{equation}
is finite.
\begin{proposition} \label{prop.SHSHpr}
Suppose that $S_h$ is defined by \eqref{def.Sh} and $S_h'$ is defined by \eqref{def.Shpr}. Then for any $\epsilon>0$ we have that
$S_h(\epsilon)$ is finite if and only if $S_h'(\epsilon)$ is finite.
\end{proposition}
\begin{proof}
Define
\[
\theta(n)^2=\int_{nh}^{(n+1)h} \|\sigma(s)\|^2_F\,ds.
\]
If $S_h(\epsilon)$ is finite, then $1-\Phi(\epsilon/\theta(n))\to 0$
as $n\to\infty$. This implies $\epsilon/\theta(n)\to\infty$ as
$n\to\infty$. Therefore by \eqref{eq.millsasy}, we have
\begin{equation} \label{eq.millsthetan}
\lim_{n\to\infty}\frac{1-\Phi(\epsilon/\theta(n))}{\theta(n)/\epsilon\cdot
\exp(-\epsilon^2/\{2\theta^2(n)\})}=\frac{1}{\sqrt{2\pi}}.
\end{equation}
Since $(1-\Phi(\epsilon/\theta(n)))_{n\geq 1}$ is summable, it
therefore follows that the sequence
\[(\theta(n)/\epsilon\cdot
\exp(-\epsilon^2/\{2\theta^2(n)\}))_{n\geq 1}
\] is summable, so $S_h'(\epsilon)$ is finite, by definition.

On the other hand, if $S_h'(\epsilon)$ is finite, and we define
$\phi:[0,\infty)\to \mathbb{R}^d$ by
\[
\phi(x)=\left\{ \begin{array}{cc} x \exp(-1/(2x^2)), & x>0, \\
0, & x=0,
\end{array}
\right.
\]
then as we have $\theta(n)\exp(-\epsilon^2/2\theta^2(n))$ summable, we have that $(\phi(\theta(n)/\epsilon))_{n\geq 1}$ is summable.
Therefore $\phi(\theta(n)/\epsilon)\to 0$ as $n\to\infty$. Then, as
$\phi$ is continuous and increasing on $[0,\infty)$, we have that
$\theta(n)/\epsilon\to 0$ as $n\to\infty$, or $\epsilon/\theta(n)\to
\infty$ as $n\to\infty$. Therefore \eqref{eq.millsthetan} holds, and
thus $(1-\Phi(\epsilon/\theta(n)))_{n\geq 1}$ is summable, which
implies that $S_h(\epsilon)$ is finite, as required.
\end{proof}

Armed with these observations, we see that the following theorem characterises the pathwise asymptotic
behaviour of solutions of \eqref{def.Y}.  In the scalar
case it yields a result of Appleby, Cheng and Rodkina in~\cite{JAJCAR:2011dresden} when $h=1$. It is also of utility
when considering the relationship between the asymptotic behaviour
of solutions of stochastic differential equations and the asymptotic
behaviour of uniform step--size discretisations.
\begin{theorem}  \label{theorem.Yclassify3}
Suppose that $\sigma$ obeys \eqref{eq.sigmacns} and $Y$ is the
unique continuous adapted process which obeys \eqref{def.Y}.
Suppose that $S_h'$ is defined by \eqref{def.Shpr}.
\begin{itemize}
\item[(A)] If
\begin{equation} \label{eq.thetastableh}
\text{$S_h'(\epsilon)$ is finite for all $\epsilon>0$},
\end{equation}
then
\begin{equation} \label{eq.Ytto0}
\lim_{t\to\infty} Y(t)=0, \quad\text{a.s.}
\end{equation}
\item[(B)] If there exists $\epsilon'>0$ such that
\begin{equation}  \label{eq.thetaboundedh}
\text{$S_h'(\epsilon)$ is finite for all $\epsilon>\epsilon'$}, \quad
\text{$S_h'(\epsilon)=+\infty$ for all $\epsilon<\epsilon'$},
\end{equation}
then there exists deterministic $0<c_1\leq c_2<+\infty$ such that
\begin{equation} \label{eq.Ytboundedh}
c_1\leq \limsup_{t\to\infty} \|Y(t)\|\leq c_2, \quad \text{a.s.}
\end{equation}
Moreover
\begin{equation} \label{eq.liminfYaveY0}
\liminf_{t\to\infty} \|Y(t)\|=0, \quad\lim_{t\to\infty}
\frac{1}{t}\int_0^t \|Y(s)\|^2\,ds=0, \quad\text{a.s.}
\end{equation}
\item[(C)] If
\begin{equation} \label{eq.thetaunstableh}
\text{$S_h'(\epsilon)=+\infty$ for all $\epsilon>0$},
\end{equation}
then
\begin{equation} \label{eq.Ytunstable}
\limsup_{t\to\infty} \|Y(t)\|=+\infty, \quad \text{a.s.}
\end{equation}
\end{itemize}
\end{theorem}
The conditions and form of Theorem~\ref{theorem.Yclassify3}, as well
as other theorems in this section, are inspired by those of
\cite[Theorem 1]{ChanWill:1989} and by \cite[Theorem 6, Corollary
7]{JAARMR:2009}.

We next show that the parameter $h>0$ in Theorem~\ref{theorem.Yclassify3}, while potentially of interest for numerical simulations, plays no role in classifying the dynamics of \eqref{def.Y}.
Therefore, we may take $h=1$ without loss of generality.
\begin{proposition} \label{prop.ShS1}
Suppose that $S_h'$ is defined by \eqref{def.Shpr}.
\begin{itemize}
\item[(i)] If  $S_1'(\epsilon)<+\infty$ for all $\epsilon>0$, then for each $h>0$ we have $S_h'(\epsilon)<+\infty$ for all $\epsilon>0$.
\item[(ii)] If there exists $\epsilon'>0$ such that $S_1'(\epsilon)<+\infty$ for all $\epsilon>\epsilon'$ and
$S_1'(\epsilon)=+\infty$ for all $\epsilon<\epsilon'$, then for each $h>0$ there exists $\epsilon_h'>0$ such that
 $S_h'(\epsilon)<+\infty$ for all $\epsilon>\epsilon_h'$ and  $S_h'(\epsilon)=+\infty$ for all $\epsilon<\epsilon_h'$.
\item[(iii)] If  $S_1'(\epsilon)=+\infty$ for all $\epsilon>0$, then for each $h>0$ we have $S_h'(\epsilon)=+\infty$ for all $\epsilon>0$.
\end{itemize}
\end{proposition}
\begin{proof}
To prove part (i), note by hypothesis that part (A) of Theorem~\ref{theorem.Yclassify3} implies $Y(t)\to 0$ as $t\to\infty$ a.s. Now suppose that there is a $h>0$ such that $S_h(\epsilon')=+\infty$ for some $\epsilon'>0$. But by parts (B) and (C) we have that $\mathbb{P}[Y(t)\to 0 \text{ as $t\to\infty$}]=0$, a contradiction.

To prove part (iii), note by hypothesis that part (C) of Theorem~\ref{theorem.Yclassify3} implies $\limsup_{t\to\infty}\|Y(t)\|=+\infty$ a.s. Now suppose that there is a $h>0$ such that $S_h(\epsilon')<+\infty$ for some $\epsilon'>0$. But by parts (A) and (B) we have that $\limsup_{t\to\infty}\|Y(t)\|<+\infty$ a.s., a contradiction.

To prove part (ii), we note by hypothesis that part (B) of Theorem~\ref{theorem.Yclassify3} implies
$0<\limsup_{t\to\infty} \|Y(t)\|<+\infty$ a.s. Suppose now there exists $h>0$ such that $S_h'(\epsilon)=+\infty$ for all $\epsilon>0$. Then we have that $\limsup_{t\to\infty} \|Y(t)\|=+\infty$ a.s., a contradiction. Suppose on the other hand that there is $h>0$ such that  $S_h'(\epsilon)<+\infty$ for all $\epsilon>0$.
Then we have that $\limsup_{t\to\infty} \|Y(t)\|=0$ a.s., a contradiction. Therefore it must follow that for each $h>0$ there exist
$\epsilon_h'', \epsilon_h'''>0$ such that $S_h'(\epsilon_h'')<+\infty$ and $S_h(\epsilon_h''')=+\infty$. Then as $\epsilon\mapsto S_h'(\epsilon)$ is a non--increasing function, it follows that for each $h>0$ there is an $\epsilon_h'>0$ such that $S_h'(\epsilon)<+\infty$ for all $\epsilon>\epsilon_h'$ and $S_h'(\epsilon)+=\infty$ for all $\epsilon<\epsilon_h'$.
\end{proof}

Combining Theorems~\ref{theorem.XYasy} and \ref{theorem.Yclassify3} we immediately get the following result concerning
the solutions of the differential equation \eqref{eq.stochlinode}.
\begin{theorem} \label{theorem.liniffsigma}
Suppose that $\sigma$ obeys \eqref{eq.sigmacns}. Let $A$ be a
$d\times d$ real matrix for which all eigenvalues have negative real
parts. Let $X$ be the solution of \eqref{eq.stochlinode} and suppose that $S_h'$ is defined by \eqref{def.Shpr}.
Then the following holds:
\begin{itemize}
\item[(A)] If $S_h'$ obeys \eqref{eq.thetastableh}, then  $\lim_{t\to\infty} X(t,\xi)=0$ a.s. for each
$\xi\in\mathbb{R}^d$;
\item[(B)] If $S_h'$ obeys \eqref{eq.thetaboundedh}, then there
exist deterministic $0<c_1\leq c_2<\infty$ independent of $\xi$ such
that
\[
 c_1\leq \limsup_{t\to\infty} \|X(t,\xi)\|\leq c_2, \quad \text{a.s.}
\]
Moreover
\begin{equation}  \label{eq.IcondnXbound}
\liminf_{t\to\infty} \|X(t,\xi)\|=0, \quad \lim_{t\to\infty}
\frac{1}{t}\int_0^t \|X(s,\xi)\|^2\,ds=0, \quad \text{a.s.}
\end{equation}
\item[(C)] If $S_h'$ obeys \eqref{eq.thetaunstableh}, then  $\limsup_{t\to\infty} \|X(t,\xi)\|=+\infty$ a.s. for each
$\xi\in\mathbb{R}^d$.
\end{itemize}
\end{theorem}
In the last case, when $S_h'(\epsilon)=+\infty$ for all $\epsilon>0$, it is interesting to ask what is the limit inferior and ergodic behaviour of $\|X(t)\|$. It is very much in the spirit of this work to ask what happens when the noise intensity fades (in some sense)
as $t\to\infty$. In cases (A) and (B), $S_h'(\epsilon)<+\infty$ for all $\epsilon$ sufficiently large. This implies that
\begin{equation}  \label{eq.sighto0}
\lim_{n\to\infty}\int_{nh}^{(n+1)h} \|\sigma(s)\|^2_F\,ds= 0.
\end{equation}
Making this additional fading noise hypothesis, we can describe more completely the limiting asymptotic behaviour of $X$ in the case when $S_h'(\epsilon)=+\infty$.
\begin{theorem} \label{theorem.liniffsigmafadeS}
Suppose that $\sigma$ obeys \eqref{eq.sigmacns}. Let $A$ be a
$d\times d$ real matrix for which all eigenvalues have negative real
parts. Let $X$ be the solution of \eqref{eq.stochlinode} and suppose that $S_h'$ is defined by \eqref{def.Shpr}.
Suppose further that \eqref{eq.sighto0} holds.
\begin{itemize}
\item[(A)] If $S_h'$ obeys \eqref{eq.thetastableh}, then  $\lim_{t\to\infty} X(t,\xi)=0$ a.s. for each
$\xi\in\mathbb{R}^d$;
\item[(B)] If $S_h'$ obeys \eqref{eq.thetaboundedh}, then there
exist deterministic $0<c_1\leq c_2<\infty$ independent of $\xi$ such
that
\[
 c_1\leq \limsup_{t\to\infty} \|X(t,\xi)\|\leq c_2, \quad \text{a.s.}
\]
Moreover, $X$ obeys \eqref{eq.IcondnXbound}.
\item[(C)] If $S_h'$ obeys \eqref{eq.thetaunstableh}, then
\[
\limsup_{t\to\infty} \|X(t,\xi)\|=+\infty, \quad\text{a.s. for each
$\xi\in\mathbb{R}^d$}
\]
Moreover,  $X$ obeys \eqref{eq.IcondnXbound}.
\end{itemize}
\end{theorem}

The condition \eqref{eq.sighto0}  is interesting because it is equivalent to asking that all solutions of \eqref{eq.stochlinode} converge to zero in mean--square.
Results yielding sufficient conditions for mean square stability of linear stochastic differential equations abound, and no claim is made for the novelty of 
the result below. However, we believe that the formulation of the result is of interest when placed in the context of our analysis of a.s. asymptotic behaviour.  
\begin{proposition} \label{prop.meansquare}
Suppose that $\sigma$ obeys \eqref{eq.sigmacns}. Let $A$ be a
$d\times d$ real matrix for which all eigenvalues have negative real
parts. Let $X$ be the solution of \eqref{eq.stochlinode}. 
Then the following are equivalent:
\begin{itemize}
\item[(A)] $\sigma$ obeys \eqref{eq.sighto0} for some $h>0$;
\item[(B)] $\sigma$ obeys \eqref{eq.sighto0} for all $h>0$;
\item[(C)] $\lim_{t\to\infty} \int_t^{t+1}\|\sigma(s)\|^2_F\,ds=0$;
\item[(D)] $\lim_{t\to\infty} \mathbb{E}[\|X(t)\|^2]=0$.
\end{itemize}
\end{proposition}

Given that the equations studied are in continuous time, it is natural to ask whether
the summation conditions can be replaced by integral conditions on $\sigma$ instead. The answer is in
the affirmative.
To this end we introduce for fixed $c>0$ the $\epsilon$--dependent
integral
\begin{equation} \label{def.Ieps}
I_c(\epsilon)=\int_0^\infty \varsigma_c(t)
\exp\left(-\frac{\epsilon^2/2}{\varsigma_c(t)^2
}\right)\chi_{(0,\infty)}\left( \varsigma_c(t)\right)\,dt,
\end{equation}
where we have defined
\begin{equation} \label{def.varsigc}
\varsigma_c(t):= \left(\int_{t}^{t+c}
\|\sigma(s)\|^2_F\,ds\right)^{1/2}, \quad t\geq 0.
\end{equation}
We notice that $\epsilon\mapsto I_c(\epsilon)$ is a monotone function,
and therefore $I_c(\cdot)$ is either finite for all $\epsilon>0$;
infinite for all $\epsilon>0$; or finite for all
$\epsilon>\epsilon'$ and infinite for all $\epsilon<\epsilon'$. The
following theorem is therefore seen to classify the asymptotic
behaviour of \eqref{def.Y}.
\begin{theorem} \label{theorem.Icondn}
Suppose that $\sigma$ obeys \eqref{eq.sigmacns} and that $Y$ is the
unique continuous adapted process which obeys \eqref{def.Y}. Let $c>0$,
$I_c(\cdot)$ be defined by \eqref{def.Ieps}, and $\varsigma_c$ by \eqref{def.varsigc}.
\begin{itemize}
\item[(A)] If
\begin{equation} \label{eq.Istable}
\text{$I_c(\epsilon)$ is finite for all $\epsilon>0$},
\end{equation}
then $\lim_{t\to\infty} Y(t)=0$ a.s.
\item[(B)] If there exists $\epsilon'>0$ such that
\begin{equation}  \label{eq.Ibounded}
\text{$I_c(\epsilon)$ is finite for all $\epsilon>\epsilon'$}, \quad
\text{$I_c(\epsilon)=+\infty$ for all $\epsilon<\epsilon'$},
\end{equation}
then there exist deterministic $0<c_1\leq c_2<+\infty$ such that
\[
c_1\leq \limsup_{t\to\infty} \|Y(t)\|\leq c_2, \quad\text{a.s.}
\]
Moreover, $Y$ also obeys  \eqref{eq.liminfYaveY0}.
\item[(C)] If
\begin{equation} \label{eq.Iunstable}
\text{$I_c(\epsilon)=+\infty$ for all $\epsilon>0$},
\end{equation}
then $\limsup_{t\to\infty} \|Y(t)\|=+\infty$ a.s.
\end{itemize}
\end{theorem}
Using this result and Theorem~\ref{theorem.XYasy}, we immediately arrive at a classification theorem for
the solution of \eqref{eq.stochlinode}.
\begin{theorem} \label{theorem.IcondnX}
Suppose that $\sigma$ obeys \eqref{eq.sigmacns} and that $X$ is the
unique continuous adapted process which obeys
\eqref{eq.stochlinode}. Suppose all the eigenvalues of $A$ have
negative real parts. Let $c>0$, $I_c(\cdot)$ be defined by
\eqref{def.Ieps}, and $\varsigma_c$ by \eqref{def.varsigc}.
\begin{itemize}
\item[(A)] If $I_c$ obeys \eqref{eq.Istable}, then
$\lim_{t\to\infty} X(t,\xi)=0$, a.s.
\item[(B)] If $I_c$ obeys \eqref{eq.Ibounded},
then there exist deterministic $0<c_1\leq c_2<+\infty$ independent of $\xi$ such that
\begin{equation*} 
c_1\leq \limsup_{t\to\infty} \|X(t,\xi)\|\leq c_2, \quad \text{a.s.}
\end{equation*}
Moreover $X$ obeys \eqref{eq.IcondnXbound}.
\item[(C)] If $I_c$ obeys \eqref{eq.Iunstable} then
$\limsup_{t\to\infty} \|X(t,\xi)\|=+\infty$ a.s.
\end{itemize}
\end{theorem}
We can prove in a manner analogous to that used to establish Proposition~\ref{prop.ShS1} that we can take $c=1$ without loss of generality in
\eqref{def.Ieps} and \eqref{def.varsigc}. It is therefore enough to consider the finiteness of $I_1(\epsilon)$ in order to determine the asymptotic behaviour.

If we impose the fading noise condition $\varsigma_c(t)\to 0$ as $t\to\infty$ (which is equivalent to \eqref{eq.sighto0}), we can demonstrate in a manner analogous to the proof of  Theorem~\ref{theorem.liniffsigmafadeS} that $\liminf_{t\to\infty} \|X(t,\xi)\|=0$ a.s. in  the case when $I_c(\epsilon)=+\infty$ for every $\epsilon>0$.
\begin{theorem} \label{theorem.IcondnXliminf}
Suppose that $\sigma$ obeys \eqref{eq.sigmacns} and that $X$ is the
unique continuous adapted process which obeys
\eqref{eq.stochlinode}. Suppose all the eigenvalues of $A$ have
negative real parts. Let $c>0$, $I_c(\cdot)$ be defined by
\eqref{def.Ieps}, and $\varsigma_c$ by \eqref{def.varsigc}. Suppose finally that $\varsigma_c(t)\to 0$ as $t\to\infty$.
\begin{itemize}
\item[(A)] If $I_c$ obeys \eqref{eq.Istable}, then
$\lim_{t\to\infty} X(t,\xi)=0$, a.s.
\item[(B)] If $I_c$ obeys \eqref{eq.Ibounded},
then there exist deterministic $0<c_1\leq c_2<+\infty$ independent of $\xi$ such that
\begin{equation*} 
c_1\leq \limsup_{t\to\infty} \|X(t,\xi)\|\leq c_2, \quad \text{a.s.}
\end{equation*}
Moreover $X$ obeys \eqref{eq.IcondnXbound}.
\item[(C)] If $I_c$ obeys \eqref{eq.Iunstable} then
$\limsup_{t\to\infty} \|X(t,\xi)\|=+\infty$ a.s.
Moreover $X$ obeys \eqref{eq.IcondnXbound}.
\end{itemize}
\end{theorem}

The result of Theorem~\ref{theorem.IcondnX} shows that $\liminf_{t\to\infty}
\|X(t)\|=0$ a.s. when $I_1(\epsilon)$ is finite for some $\epsilon>0$
and infinite for others. In Theorem~\ref{theorem.IcondnXliminf} we strengthened the condition on the smallness of the noise
coefficient, enabling us to prove that when $I_1(\epsilon)=+\infty$ for
every $\epsilon>0$, we have $\limsup_{t\to\infty} \|X(t)\|=+\infty$ and $\liminf_{t\to\infty} \|X(t)\|=0$ a.s.
We now give an example which shows that this conclusion cannot be extended if the diffusion coefficient grows in intensity as $t\to\infty$,
and that therefore part (C) of Theorem~\ref{theorem.IcondnX} is the most general conclusion that can be drawn without imposing more specific
growth conditions on the diffusion coefficient.
\begin{example} \label{examp.liminf0fininfty}
Suppose that $d=r\geq 3$, that $A=-I_d$ and that
$\sigma(t)=\eta(t)I_d$ for $t\geq 0$, where $\eta\in
C([0,\infty);(0,\infty))$. Suppose also that
\[
\lim_{t\to\infty} \int_0^t e^{2s}\eta^2(s)\,ds=+\infty.
\]
Then the $i$--th component of $X$ obeys
\[
X_i(t)=\xi_i e^{-t}+e^{-t}\int_0^t e^s\eta(s)\,dB_i(s), \quad t\geq
0.
\]
Hence
\[
e^{2t} \|X(t)\|_2^2=\|\xi\|^2_2+\sum_{i=1}^d \left(\int_0^t
e^s\eta(s)\,dB_i(s)\right)^2, \quad t\geq 0.
\]
Define
\[
T(t):=\int_0^t e^{2s}\eta^2(s)\,ds, \quad t\geq 0.
\]
Then $T:[0,\infty)\to [0,\infty)$ is an increasing and $C^1$
function with $T(t)\to\infty$ as $t\to\infty$. Define
$\tau(t)=T^{-1}(t)$ for $t\geq 0$ and
\[
U(t)=\|\xi\|^2_2+\sum_{i=1}^d \left(\int_0^t
e^s\eta(s)\,dB_i(s)\right)^2, \quad t\geq 0.
\]
Also define $\tilde{U}(t)=U(\tau(t))$ and
\[
B_i^\ast(t)=\int_0^{\tau(t)} e^s\eta(s)\,dB_i(s), \quad t\geq 0.
\]
Let $\mathcal{G}(t)=\mathcal{F}^B(\tau(t))$. Then $\tilde{U}$ and
$B_i^\ast$ are $\mathcal{G}$--adapted and
\[
\tilde{U}(t)=\|\xi\|^2_2+\sum_{i=1}^d B_i^\ast(t)^2, \quad t\geq 0.
\]
We now establish that $B_i^\ast$ is a $\mathcal{G}$ standard
Brownian motion. To do this we must check the conditions of L\'evy's
theorem for characterising standard Brownian motion. First, we see
that $B_i^\ast$ is $\mathcal{F}^B(\tau(t))$ measurable, and
therefore $\mathcal{G}(t)$ measurable. Since $\tau$ is increasing,
$\mathcal{G}$ is a filtration. Also because $\tau$ is continuous and
$s\mapsto e^s\eta(s)$ is continuous, then $t\mapsto B_i^\ast(t)$ is
continuous. Finally, if we let $I_i(t)=\int_0^t e^s\eta(s)dB_i(s)$,
then $\mathbb{E}[I_i(t)^2]=\int_0^t e^{2s}\eta(s)^2ds=T(t)$. Thus
\[
\mathbb{E}[B_i^\ast(t)^2]=\mathbb{E}[I_i(\tau(t))^2]=T(\tau(t))=t<+\infty.
\]
Therefore, we need only to check that $B_i^\ast$ obeys the
projection property for martingales. Let $t>s\geq0$. Then as $\tau$
is increasing, we have
\begin{align*}
\mathbb{E}[B_i^\ast(t)|\mathcal{G}(s)]&=\mathbb{E}[I_i(\tau(t))|\mathcal{F}^B(\tau(s))]\\
&=\mathbb{E}[\int_{\tau(s)}^{\tau(t)}e^u\eta(u)dB_i(u)+B_i^\ast(s)|\mathcal{F}^B(\tau(s))]\\
&=\mathbb{E}[\int_{\tau(s)}^{\tau(t)}e^u\eta(u)dB_i(u)|\mathcal{F}^B(\tau(s))]+B_i^\ast(s)\\
&=\mathbb{E}[\int_{\tau(s)}^{\tau(t)}e^u\eta(u)dB_i(u)]+B_i^\ast(s)=B_i^\ast(s).
\end{align*}
Hence $B^\ast_i$ is a $\mathcal{G}(t)$--martingale. Finally,
$\langle
B_i^\ast\rangle(t)=\int_0^{\tau(t)}e^{2s}\eta(s)^2ds=T(\tau(t))=t$.
Therefore, by L\'evy's characterisation theorem, $B_i^\ast$ is a
$\mathcal{G}$ standard Brownian motion.
Also, because the Brownian motions $B_1, \ldots, B_d$ are
independent, it follows that $B_1^\ast, B_2^\ast, \ldots, B_d^\ast$
are independent $\mathcal{G}$--adapted standard Brownian motions.
Therefore $\tilde{U}$ is a $d$--dimensional square Bessel process
starting at $\|\xi\|^2_2$, and indeed
\[
e^{2\tau(t)}\|X(\tau(t))\|^2_2=\tilde{U}(t), \quad t\geq 0.
\]
Thus, $\tilde{U}_2=\sqrt{\tilde{U}}$ is a $d$--dimensional Bessel
process starting at $\|\xi\|_2$.

Now, if $\xi\neq 0$, it was proven in Appleby and
Wu~\cite{ApplebyWu:2009} that
\[
\liminf_{t\to\infty} \frac{\log
\frac{\tilde{U}_2(t)}{\sqrt{t}}}{\log\log t}=-\frac{1}{d-2}, \quad
\limsup_{t\to\infty} \frac{\tilde{U}_2(t)}{\sqrt{2t\log\log t}}=1,
\quad\text{a.s.}
\]
Hence
\[
\liminf_{t\to\infty} \frac{\log
\frac{e^{\tau(t)}\|X(\tau(t))\|_2}{\sqrt{t}}}{\log\log
t}=-\frac{1}{d-2}, \quad \text{a.s.}
\]
which yields
\begin{equation}\label{eq.liminfXbesselexample}
\liminf_{t\to\infty} \frac{\log
\frac{\|X(t)\|_2}{\sqrt{e^{-2t}T(t)}}}{\log\log
T(t)}=-\frac{1}{d-2}, \quad \limsup_{t\to\infty}
\frac{\|X(t)\|_2}{\sqrt{2e^{-2t}T(t)\log\log T(t)}}=1,
\quad\text{a.s.}
\end{equation}
If we suppose that $\eta$ is such that $\eta'(t)/\eta(t)\to 0$ as
$t\to\infty$, so that $\eta$ neither decays nor grows at an
exponential rate, we have by l'H\^opital's rule that
\[
\lim_{t\to\infty} \frac{T(t)}{e^{2t}\eta(t)^2}=\frac{1}{2},
\]
and because $\lim_{t\to\infty} \log\eta(t)/t=0$, we have also that
\[
\lim_{t\to\infty} \frac{\log\log T(t)}{\log t}=1.
\]
Therefore, from \eqref{eq.liminfXbesselexample} we get
\[
\liminf_{t\to\infty} \frac{\log \frac{\|X(t)\|_2}{\frac{1}{\sqrt{2}}
\eta(t)}}{\log t}=-\frac{1}{d-2}, \quad \limsup_{t\to\infty}
\frac{\|X(t)\|_2}{\sqrt{\eta^2(t)\log t }}=1, \quad\text{a.s.}
\]

Now, we suppose that $\eta(t)/t^\alpha\to L\in (0,\infty)$ as
$t\to\infty$. If $\alpha\geq 0$, we can show that all the hypotheses
hold and that $I_1(\epsilon)=+\infty$ for all $\epsilon>0$. Moreover,
if $\alpha>1/(d-2)>0$, then
\[
\lim_{t\to\infty} \|X(t)\|_2=+\infty, \quad \text{a.s.}
\]
while if $0\leq \alpha<1/(d-2)$, we have
\[
\liminf_{t\to\infty} \|X(t)\|_2= 0, \quad \limsup_{t\to\infty}
\|X(t)\|_2=+\infty, \quad \text{a.s.}
\]
(In the case $\alpha<0$, we have that $X(t)\to 0$ as $t\to\infty$
a.s. because $I_1(\epsilon)$ is finite for all $\epsilon>0$.)

Therefore, it can be seen that without further information on the
growth or decay rate of $\|\sigma(t)\|$ as $t\to\infty$, it is
impossible to make a general conclusion about the size of
$\liminf_{t\to\infty} \|X(t)\|$. In this sense, the overall
conclusions of Theorem~\ref{theorem.Icondn}  cannot be improved upon
if $d\geq 3$ without further analysis.

However, in the case when $d=1$ (and one can take $r=1$
without loss of generality), we can show that $\liminf_{t\to\infty} |X(t)|=0$ a.s.
Suppose that $I_1(\epsilon)=+\infty$ for all $\epsilon>0$. Then $\limsup_{t\to\infty} |X(t)|=+\infty$.
By Theorem~\ref{theorem.XYasy}, it follows that $\limsup_{t\to\infty} |Y(t)|=+\infty$ a.s. Then we know
that $S_1(\epsilon)=+\infty$ for all $\epsilon>0$. Hence $\sigma^2\not\in L^1(0,\infty)$. By mimicking a proof of
a result in Appleby, Cheng and Rodkina~\cite{JAJCAR:2011dresden}, it follows that we must have $\liminf_{t\to\infty} |X(t)|=0$ a.s.

%
\end{example}

We now present a result concerning the inability of noise to
stabilise the asymptotically stable differential equation
$x'(t)=Ax(t)$.
\begin{theorem} \label{theorem.nonstabilise}
Suppose that $\sigma$ obeys \eqref{eq.sigmacns} and that
$X(\cdot,\xi)$ is the unique continuous adapted process which obeys
\eqref{eq.stochlinode} with initial condition $X(0)=\xi$. Then the
following are equivalent:
\begin{itemize}
\item[(A)] All the eigenvalues of $A$ have negative real
parts, and $I$ defined by \eqref{def.Ieps}  obeys
\eqref{eq.Istable};
\item[(B)] There is a basis $(\xi_i)_{i=1}^d$ of $\mathbb{R}^d$ and an event $C$ with $\mathbb{P}[C]>0$ given by
\[
C=\{\omega: \lim_{t\to\infty} X(t,\xi_i,\omega)=0, \text{ for
$i=1,\ldots,d$, } \lim_{t\to\infty} X(t,0,\omega)=0\};
\]
\item[(C)] For each $\xi\in \mathbb{R}^d$ we have $\lim_{t\to\infty} X(t,\xi)=0$ a.s.
\end{itemize}
\end{theorem}

This section closes with one further remark. The classification of the asymptotic behaviour of \eqref{eq.stochlinode} is achieved by means of summability or
equivalent integrability conditions which are written in terms of the Frobenius norm of $\sigma$. However, by norm equivalence, it can be shown that \text{any}
norm on $\mathbb{R}^{d\times r}$ can be used in place of the Frobenius norm. More precisely, the following holds.
\begin{proposition} \label{prop.normequiv}
Let $\|\cdot\|$ be any norm on $\mathbb{R}^{d\times r}$, and define
\begin{gather*}
J_1(\epsilon)=\int_0^\infty \sqrt{\int_t^{t+1}\|\sigma(s)\|^2\,ds} \exp\left(-\frac{\epsilon^2}{2\int_t^{t+1}\|\sigma(s)\|^2\,ds}\right)\,dt, \\
T_1'(\epsilon)=\sum_{n=1}^\infty \left\{\sqrt{\int_n^{n+1}\|\sigma(s)\|^2\,ds} \exp\left(-\frac{\epsilon^2}{2\int_n^{n+1}\|\sigma(s)\|^2\,ds}\right) \right\}.
\end{gather*}
Let $S_1'$ be defined by \eqref{def.Shpr} and $I_1$ be defined by \eqref{def.Ieps} and \eqref{def.varsigc}.
\begin{itemize}
\item[(A)] The following statements are equivalent:
\begin{enumerate}
\item[(i)] $S_1'(\epsilon)<+\infty$ for all $\epsilon>0$;
\item[(ii)] $T_1'(\epsilon)<+\infty$ for all $\epsilon>0$;
\item[(iii)] $I_1(\epsilon)<+\infty$ for all $\epsilon>0$;
\item[(iv)] $J_1(\epsilon)<+\infty$ for all $\epsilon>0$.
\end{enumerate}
\item[(ii)] The following statements are equivalent:
\begin{enumerate}
\item[(i)] There exists $\epsilon_1>0$ such that $S_1'(\epsilon)<+\infty$ for all $\epsilon>\epsilon_1$ and $S_1'(\epsilon)=+\infty$ for all $\epsilon<\epsilon_1$;
\item[(ii)] There exists $\epsilon_2>0$ such that $T_1'(\epsilon)<+\infty$ for all $\epsilon>\epsilon_2$ and $T_1'(\epsilon)=+\infty$ for all $\epsilon<\epsilon_2$;
\item[(iii)] There exists $\epsilon_3>0$ such that $I_1(\epsilon)<+\infty$ for all $\epsilon>\epsilon_3$ and $I_1(\epsilon)=+\infty$ for all
$\epsilon<\epsilon_3$;
\item[(iv)] There exists $\epsilon_4>0$ such that $J_1(\epsilon)<+\infty$ for all $\epsilon>\epsilon_4$ and $J_1(\epsilon)=+\infty$ for all
$\epsilon<\epsilon_4$;
\end{enumerate}
\item[(iii)]  The following statements are equivalent:
\begin{enumerate}
\item[(i)] $S_1'(\epsilon)=+\infty$ for all $\epsilon>0$;
\item[(ii)] $T_1'(\epsilon)=+\infty$ for all $\epsilon>0$;
\item[(iii)] $I_1(\epsilon)=+\infty$ for all $\epsilon>0$;
\item[(iv)] $J_1(\epsilon)=+\infty$ for all $\epsilon>0$.
\end{enumerate}
\end{itemize}
\end{proposition}
The proof is straightforward and hence omitted.

\section{Sufficient conditions on $\sigma$ for asymptotic classification}
Although the summability conditions on \eqref{def.Sh} classify  necessary and
sufficient, it can be quite difficult to check in practice. We
supply more easily--checked sufficient conditions on $\sigma$ for which
the solution of \eqref{eq.stochlinode} converges to zero, is bounded or is
unbounded.

It is well--known in the case that all eigenvalues of $A$ have negative real parts that the solution of \eqref{eq.stochlinode} is a.s. asymptotically stable
in the case that $\sigma\in L^2([0,\infty);\mathbb{R}^{d\times r})$. We can see that this fact is a simple corollary of parts (A) of Theorem~\ref{theorem.XYasy},
Theorem~\ref{theorem.Yclassify3} and the following observation.
\begin{proposition} \label{prop.siginL2}
If $\sigma\in L^2([0,\infty);\mathbb{R}^{d\times r})$, then $S_1'(\epsilon)<+\infty$ for every $\epsilon>0$.
\end{proposition}
\begin{proof}
By hypothesis, we have that $\int_{n}^{n+1} \|\sigma(s)\|_F^2\,ds  \to 0$ as $n\to\infty$. Since
\[
\lim_{x\to \infty} \frac{x^{-1}e^{-x^2/2}}{x^{-2}}=0,
\]
we have for each $\epsilon>0$ that
\[
\lim_{n\to \infty} \frac{\sqrt{\int_{n}^{n+1} \|\sigma(s)\|_F^2\,ds}\exp\left(-\frac{\epsilon^2}{2\int_{n}^{n+1} \|\sigma(s)\|_F^2\,ds}\right)}{\int_{n}^{n+1} \|\sigma(s)\|_F^2\,ds}=0.
\]
Since the denominator is a summable sequence, the numerator must also be summable; and this is simply the statement that
$S_1'(\epsilon)<+\infty$ for every $\epsilon>0$, as required.
\end{proof}

The next result characterises the asymptotic behaviour of $X$, according as to whether a certain limit exists, and is zero, finite but non--zero,
or infinite.
\begin{theorem} \label{theorem.XsiglogtL}
Suppose that $\sigma$ obeys \eqref{eq.sigmacns}. Suppose that there exists $h>0$ and $L_h\in[0,\infty]$ such that
\begin{equation} \label{eq.intsiglogntoLh}
\lim_{n\to\infty} \int_{nh}^{(n+1)h} \|\sigma(s)\|_F^2\,ds \cdot \log n =L_h.
\end{equation}
Let $A$ be a $d\times d$ real matrix whose eigenvalues all have negative real parts, and suppose that $X$ is
the unique continuous adapted process which obeys \eqref{eq.stochlinode}.
\begin{itemize}
\item[(i)] If $L_h=0$, then $\lim_{t\to\infty} X(t,\xi)=0$, a.s.
\item[(ii)] If $L_h\in (0,\infty)$, then there exist $0\leq c_1\leq c_2<\infty$ independent of $\xi$ such that
\[
c_2\leq \limsup_{t\to\infty} \|X(t,\xi)\|\leq c_2, \quad \liminf_{t\to\infty} \|X(t,\xi)\|=0, \quad \text{a.s.}
\]
and
\[
\lim_{t\to\infty} \frac{1}{t}\int_0^t \|X(s,\xi)\|^2\,ds =0,\quad\text{a.s.}
\]
\item[(iii)] If $L_h=+\infty$, then $\limsup_{t\to\infty} \|X(t,\xi)\|=+\infty$ a.s.
\end{itemize}
\end{theorem}
If pointwise conditions are preferred  to \eqref{eq.intsiglogntoLh} in Theorem~\ref{theorem.XsiglogtL}, we may instead impose the condition
\[
\lim_{t\to\infty} \|\sigma(t)\|^2_F\log t=L\in [0,\infty]
\]
on $\sigma$. In this case, if $L=0$, then $L_h=0$ in \eqref{eq.intsiglogntoLh}, and part (i) of Theorem~\ref{theorem.XsiglogtL} applies; if $L\in (0,\infty)$, then $L_h=hL$ in \eqref{eq.intsiglogntoLh} and part (ii) of
Theorem~\ref{theorem.XsiglogtL} applies; and if $L=\infty$, then $L_h=+\infty$ in \eqref{eq.intsiglogntoLh},
and part (iii) of Theorem~\ref{theorem.XsiglogtL} applies.

We can also characterise the asymptotic stability of solutions of solutions with a very simple condition, contingent on a certain class of
monotonicity conditions holding on $t\mapsto \|\sigma(t)\|$.
\begin{theorem} \label{theorem.Xto0iff}
Suppose that $\sigma$ obeys \eqref{eq.sigmacns}. Let $A$ be a $d\times d$ real matrix all of whose eigenvalues have negative real parts. Let $X$ be the unique continuous adapted process which obeys \eqref{eq.stochlinode}.
Suppose that there is $h>0$ such that the sequence $n\mapsto \int_{nh}^{(n+1)h} \|\sigma(s)\|^2_F\,ds$ is non--increasing. Then the following are equivalent.
\begin{itemize}
\item[(A)] $\lim_{n\to\infty} \int_{nh}^{(n+1)h} \|\sigma(s)\|^2_F\,ds \cdot \log n = 0$;
\item[(B)] $\lim_{t\to\infty} X(t,\xi)=0$ a.s. for each $\xi\in \mathbb{R}^d$;
\item[(C)] There exists $\xi\in \mathbb{R}^d$ such that $\lim_{t\to\infty} X(t,\xi)=0$ with positive probability.
\end{itemize}
\end{theorem}

Stronger monotonicity conditions which can be imposed are that
\[
t\mapsto \Sigma^2_1(t)=\int_t^{t+1}\|\sigma(s)\|^2_F\,ds, \quad t\mapsto \Sigma^2_1(t)=\|\sigma(t)\|^2_F,
\]
are non--increasing. In this case statement (A) in Theorem~\ref{theorem.Xto0iff} can be replaced by
\[
\lim_{t\to\infty} \Sigma_i^2(t)\log t=0.
\]

\section{Periodic Affine Equations}
We present one further application of our results, which enables a classification of the asymptotic behaviour of affine stochastic differential equations with 
periodic features to be analysed. Towards this end, suppose that 
\begin{equation} \label{def.Amat}
A\in C([0,\infty);\mathbb{R}^{d\times d} \text{ is a $T$--periodic function}
\end{equation}
and consider the stochastic differential equation
\begin{equation} \label{eq.sdeperiod}
dX(t)=A(t)X(t)\,dt + \sigma(t)\,dB(t), \quad t\geq 0; \quad X(0)=\xi\in \mathbb{R}^d
\end{equation}
where as before $\sigma\in C([0,\infty);\mathbb{R}^{d\times r})$ and $B$ is an $r$--dimensional standard Brownian motion. It is standard that 
there is a unique continuous adapted process which obeys \eqref{eq.sdeperiod}. 

The analysis of \eqref{eq.sdeperiod} is facilitated greatly by the introduction of the unique continuous $\mathbb{R}^{d\times d}$--valued solution of 
\begin{equation} \label{eq.fundsolnperiod}
\Psi'(t)=A(t)\Psi(t), \quad t\geq 0, \quad \Psi(0)=I_d.
\end{equation}
In general 
\[
\text{det}(\Psi(t))=\exp\left(\int_0^t \text{tr}(A(s)))\,ds\right)\neq 0, \quad t\geq 0, 
\]
so $\Psi(t)$ is invertible for all $t\geq 0$.
The matrix $\Psi(T)$ plays a central role in the asymptotic theory of \eqref{eq.fundsolnperiod} and \eqref{eq.sdeperiod}. It is called the Floquet multiplier. 
Let us assume that 
\begin{equation} \label{eq.floqmultlt1}
\rho(\Psi(T))<1
\end{equation}
where $\rho(C)$ denotes the spectral radius of the square matrix $C$. 

\begin{theorem} \label{theorem.stochfloq}
Suppose that $\sigma$ obeys \eqref{eq.sigmacns} and $A$ obeys \eqref{def.Amat}. 
Suppose that the fundamental solution $\Psi$ of \eqref{eq.fundsolnperiod} is such that $\rho(\Psi(T))<1$.
Let $X$ be the solution of \eqref{eq.sdeperiod} and suppose that $S_h'$ is defined by \eqref{def.Shpr}.
Then the following holds:
\begin{itemize}
\item[(A)] If $S_h'$ obeys \eqref{eq.thetastableh}, then  $\lim_{t\to\infty} X(t,\xi)=0$ a.s. for each
$\xi\in\mathbb{R}^d$;
\item[(B)] If $S_h'$ obeys \eqref{eq.thetaboundedh}, then there
exist deterministic $0<c_1\leq c_2<\infty$ independent of $\xi$ such
that
\[
 c_1\leq \limsup_{t\to\infty} \|X(t,\xi)\|\leq c_2, \quad \text{a.s.}
\]
Moreover
\begin{equation*}  
\liminf_{t\to\infty} \|X(t,\xi)\|=0, \quad \lim_{t\to\infty}
\frac{1}{t}\int_0^t \|X(s,\xi)\|^2\,ds=0, \quad \text{a.s.}
\end{equation*}
\item[(C)] If $S_h'$ obeys \eqref{eq.thetaunstableh}, then  $\limsup_{t\to\infty} \|X(t,\xi)\|=+\infty$ a.s. for each
$\xi\in\mathbb{R}^d$.
\end{itemize}
\end{theorem}
In the case that $S_h'(\epsilon)=+\infty$ for all $\epsilon>0$, but $\sigma$ obeys the fading noise condition \eqref{eq.sighto0}, 
we can refine the asymptotic result in a manner identical to that in Theorem~\ref{theorem.liniffsigmafadeS} in the autonomous case.
\begin{theorem} \label{theorem.stochfloqsmallnoise}
Suppose that $\sigma$ obeys \eqref{eq.sigmacns} and $A$ obeys \eqref{def.Amat}. 
Suppose that the fundamental solution $\Psi$ of \eqref{eq.fundsolnperiod} is such that $\rho(\Psi(T))<1$.
Let $X$ be the solution of \eqref{eq.sdeperiod} and suppose that $S_h'$ is defined by \eqref{def.Shpr}.
Suppose further that $\sigma$ obeys \eqref{eq.sighto0}.
Then the following holds:
\begin{itemize}
\item[(A)] If $S_h'$ obeys \eqref{eq.thetastableh}, then  $\lim_{t\to\infty} X(t,\xi)=0$ a.s. for each
$\xi\in\mathbb{R}^d$;
\item[(B)] If $S_h'$ obeys \eqref{eq.thetaboundedh}, then there
exist deterministic $0<c_1\leq c_2<\infty$ independent of $\xi$ such
that
\[
 c_1\leq \limsup_{t\to\infty} \|X(t,\xi)\|\leq c_2, \quad \text{a.s.}
\]
Moreover
\begin{equation*}  
\liminf_{t\to\infty} \|X(t,\xi)\|=0, \quad \lim_{t\to\infty}
\frac{1}{t}\int_0^t \|X(s,\xi)\|^2\,ds=0, \quad \text{a.s.}
\end{equation*}
\item[(C)] If $S_h'$ obeys \eqref{eq.thetaunstableh}, then  $\limsup_{t\to\infty} \|X(t,\xi)\|=+\infty$ a.s. for each
$\xi\in\mathbb{R}^d$. Moreover
\begin{equation*}  
\liminf_{t\to\infty} \|X(t,\xi)\|=0, \quad \lim_{t\to\infty}
\frac{1}{t}\int_0^t \|X(s,\xi)\|^2\,ds=0, \quad \text{a.s.}
\end{equation*}
\end{itemize}
\end{theorem}

\section{A Key Theorem} 
The main results concerning the asymptotic behaviour of $Y$ in this paper (namely Theorems~\ref{theorem.Yclassify3} and \ref{theorem.Icondn}) are corollaries of
a key technical result, which is stated and proven in this section.

Suppose that $(t_n)_{n\geq 0}$ is an increasing sequence with
$t_0=0$ and $\lim_{n\to\infty} t_n=+\infty$. Define
\begin{equation} \label{def.St}
S_{t_\cdot}(\epsilon) = \sum_{n=0}^\infty \left\{
1-\Phi\left(\frac{\epsilon}{\sqrt{\int_{t_n}^{t_{n+1}}\|\sigma(s)\|_F^2\,ds}}
\right) \right\}.
\end{equation}
If there are uniform upper and lower bounds on the spacing of the sequence, it transpires that the finiteness (or otherwise) of the sum enables us to characterise the long run behaviour of \eqref{def.Y}.
The following theorem then characterises the pathwise asymptotic
behaviour of solutions of \eqref{def.Y}.
\begin{theorem}  \label{theorem.Yclassify2}
Suppose that $\sigma$ obeys \eqref{eq.sigmacns} and that $Y$ is the
unique continuous adapted process which obeys \eqref{def.Y}. Let
$S_{t_\cdot}(\epsilon)$ be defined by \eqref{def.St} where $t_\cdot$
is any $\epsilon$--independent sequence obeying
\begin{equation} \label{prop.tseq} t_0=0, \quad
0<\alpha\leq t_{n+1}-t_n\leq \beta<+\infty, \quad \lim_{n\to\infty}
t_n=+\infty
\end{equation}
for some $0<\alpha\leq \beta<+\infty$.
\begin{itemize}
\item[(A)] If
\begin{equation} \label{eq.thetastablet}
\text{$S_{t_\cdot}(\epsilon)$ is finite for all $\epsilon>0$},
\end{equation}
then $\lim_{t\to\infty} Y(t)=0$ a.s.
\item[(B)]
\begin{enumerate}
\item[(i)] If there exists $\epsilon'>0$ such that
\begin{equation}  \label{eq.thetaboundedt1}
\text{$S_{t_\cdot}(\epsilon)$ is finite for all
$\epsilon>\epsilon'$},
\end{equation}
then there exists a deterministic $0<c_2<+\infty$ such that
\begin{equation*} 
\limsup_{t\to\infty} \|Y(t)\|\leq c_2, \quad \text{a.s.}
\end{equation*}
\item[(ii)] On the other hand, if there exists $\epsilon''>0$ such that
\begin{equation}  \label{eq.thetaboundedt2}
\text{$S_{\tau_\cdot}(\epsilon)=+\infty$ for all
$\epsilon<\epsilon''$},
\end{equation}
where $\tau$ is any $\epsilon$--independent sequence obeying
\eqref{prop.tseq}, then there exists a deterministic $0<c_1<+\infty$
such that
\begin{equation*} 
\limsup_{t\to\infty} \|Y(t)\|\geq c_1, \quad \text{a.s.}
\end{equation*}
\end{enumerate}
\item[(C)] If
\begin{equation} \label{eq.thetaunstablet}
\text{$S_{t_\cdot}(\epsilon)=+\infty$ for all $\epsilon>0$},
\end{equation}
then
$\limsup_{t\to\infty} \|Y(t)\|=+\infty$ a.s.
\end{itemize}
\end{theorem}

\subsection{Proof of Theorem~\ref{theorem.Yclassify2}: preliminary estimates}
We start by showing how estimates on the rows of the matrix $\sigma$
relate to its Frobenius norm. Let $(t_n)_{n\geq 0}$ is an increasing
sequence with $t_0=0$ and $\lim_{n\to\infty} t_n=+\infty$ and
define, by analogy to \eqref{def.St},
\begin{equation} \label{def.S1t}
S^{1}_{t\cdot}(\epsilon) = \sum_{n=0}^\infty \sum_{i=1}^d \left\{
1-\Phi\left(\frac{\epsilon}{\sqrt{\int_{t_n}^{t_{n+1}} \sum_{j=1}^r
\sigma_{ij}^2(s)\,ds}} \right) \right\}.
\end{equation}
Define
\begin{gather}\label{def.thetat}
\theta^2(n)=\int_{t_n}^{t_{n+1}} \|\sigma(s)\|_F^2\,ds, \\
\label{def.thetait} \theta_i^2(n)=\int_{t_n}^{t_{n+1}} \sum_{j=1}^r
\sigma_{ij}^2(s)\,ds, \quad i=1,\ldots,d. \end{gather} We can see
that as $S^1_{t\cdot}$ is a monotone function of $\epsilon$, it is
the case that either (i) $S^1_{t\cdot}(\epsilon)$ is finite for all
$\epsilon>0$; (ii) there is $\epsilon_1'>0$ such that for all
$\epsilon>\epsilon'_1$ we have $S^1_{t\cdot}(\epsilon)<+\infty$ and
$S^1_{t\cdot}(\epsilon)=+\infty$ for all $\epsilon<\epsilon'_1$; and
(iii) $S^1_{t\cdot}(\epsilon)=+\infty$ for all $\epsilon>0$. In the
next lemma, we show that $S_{t\cdot}$ defined by \eqref{def.St} is
always finite if and only if $S^1_{t\cdot}$ is; that $S_{t\cdot}$ is
infinite if and only if $S^1_{t\cdot}$ is; and that $S_{t\cdot}$ and
$S^1_{t\cdot}$ are sometimes finite and sometimes infinite only if
the other is.
\begin{lemma} \label{lemma.SS1}
Let $(t_n)_{n\geq 0}$ be an increasing sequence with $t_0=0$ and
$\lim_{n\to\infty} t_n=+\infty$. Suppose that $S_{t_\cdot}$ is
defined by \eqref{def.St} and that $S^1_{t\cdot}$ is defined by
\eqref{def.S1t}.
\begin{itemize}
\item[(a)]  The following are equivalent:
\begin{enumerate}
\item[(i)] $S_{t_\cdot}(\epsilon)<+\infty$ for all $\epsilon>0$;
\item[(ii)] $S^1_{t_\cdot}(\epsilon)<+\infty$ for all $\epsilon>0$.
\end{enumerate}
\item[(b)] The following are equivalent:
\begin{enumerate}
\item[(i)] There exists $\epsilon'>0$ such  that for all $\epsilon>\epsilon'$ we
have $S_{t_\cdot}(\epsilon)<+\infty$ and
$S_{t_\cdot}(\epsilon)=+\infty$ for all $\epsilon<\epsilon'$;
\item[(ii)] There exists $\epsilon_1'>0$ such  that for all $\epsilon>\epsilon_1'$ we
have $S^1_{t_\cdot}(\epsilon)<+\infty$ and
$S^1_{t_\cdot}(\epsilon)=+\infty$ for all $\epsilon<\epsilon_1'$;
\end{enumerate}
\item[(c)] The following are equivalent:
\begin{enumerate}
\item[(i)] $S_{t_\cdot}(\epsilon)=+\infty$ for all $\epsilon>0$;
\item[(ii)] $S^1_{t_\cdot}(\epsilon)=+\infty$ for all $\epsilon>0$.
\end{enumerate}
\end{itemize}
\end{lemma}
\begin{proof}
With $\theta$ and $\theta_i$ defined by \eqref{def.thetat} and
\eqref{def.thetait}, we have $\theta^2(n)\geq \theta_i(n)^2$ for
each $i=1,\ldots,d$. Thus
\begin{equation} \label{eq.phi1}
\sum_{i=1}^d
\left\{1-\Phi\left(\frac{\epsilon}{\theta_i(n)}\right)\right\}\leq d
\left(1-\Phi\left(\frac{\epsilon}{\theta(n)}\right)\right).
\end{equation}
Suppose, for each $n$, that $Z_i(n)$ for $i=1,\ldots,d$ are
independent standard normal random variables. Define
$Z(n)=(Z_1(n),Z_2(n),\ldots,Z_d(n))$ and suppose that $(Z(n))_{n\geq
0}$ are a sequence of independent normal vectors. Define finally
\[
X_i(n)=\theta_i(n)Z_i(n), \quad X(n)=\sum_{i=1}^d X_i(n), \quad
n\geq 0.
\]
Then we have that $X_i$ is a zero mean normal with variance
$\theta_i^2$ and $X$ is a zero mean normal with variance $\theta^2$.
Define $Z^\ast(n)=X(n)/\theta(n)$ is a standard normal random
variable. Therefore we have that
\begin{equation} \label{eq.thth11}
\mathbb{P}[|X(n)|>\epsilon]=\mathbb{P}[|Z^\ast(n)|\geq
\epsilon/\theta(n)]=2\mathbb{P}[Z^\ast(n)\geq \epsilon/\theta(n)]
=2\left(1-\Phi\left(\frac{\epsilon}{\theta(n)}\right)\right).
\end{equation}
With $A_i(n)=\{ |X_i(n)|\leq \epsilon/d\}$, $B(n)=\{\sum_{i=1}^d
|X_i(n)|\leq \epsilon\}$, then $\cap_{i=1}^d A_i(n) \subseteq B(n)$,
so
\[
\mathbb{P}\left[|X(n)|>\epsilon\right]\leq
\mathbb{P}[\overline{B}(n)]\leq
\mathbb{P}\left[\overline{\cap_{i=1}^d A_i(n)}\right]
=\mathbb{P}\left[\cup_{i=1}^d \overline{A_i(n)}\right] \leq
\sum_{i=1}^d\mathbb{P}\left[\overline{A_i(n)}\right].
\]
Since $X_i=\theta_i Z_i$, we have \begin{equation} \label{eq.thth12}
\mathbb{P}\left[|X(n)|>\epsilon\right]\leq
2\sum_{i=1}^d\mathbb{P}\left[X_i(n)\geq \epsilon/d\right]
=2\sum_{i=1}^d \left\{ 1-\Phi\left(\frac{\epsilon/d}{\theta_i(n)}
\right) \right\}.
\end{equation}
By \eqref{eq.thth11} and \eqref{eq.thth12}, we get
\begin{equation} \label{eq.phi2}
1-\Phi\left(\frac{\epsilon}{\theta(n)}\right) \leq \sum_{i=1}^d
\left\{ 1-\Phi\left(\frac{\epsilon/d}{\theta_i(n)} \right) \right\}.
\end{equation}
From \eqref{eq.phi1}, we can see that
$S_{t_\cdot}(\epsilon)<+\infty$ implies that
$S^1_{t_\cdot}(\epsilon)<+\infty$ and from \eqref{eq.phi2} that
$S^1_{t_\cdot}(\epsilon/d)<+\infty$ implies
$S_{t_\cdot}(\epsilon)<+\infty$. Therefore, we have that part (a)
holds. Part (c) holds similarly, because from \eqref{eq.phi1} we
have that $S^1_{t_\cdot}(\epsilon)=+\infty$ implies
$S_{t_\cdot}(\epsilon)=+\infty$, and from \eqref{eq.phi2} we have
that $S^1_{t_\cdot}(\epsilon/d)=+\infty$ implies
$S_{t_\cdot}(\epsilon)=+\infty$. As to the proof of part (b),
suppose that (i) holds. Then by \eqref{eq.phi1}, we can see that
$S^1_{t_\cdot}(\epsilon)\leq S_{t_\cdot}(\epsilon)<+\infty$ for all
$\epsilon<\epsilon'$, and by \eqref{eq.phi2} that
$S^1_{t_\cdot}(\epsilon/d)\geq S(\epsilon)=+\infty$ for all
$\epsilon<\epsilon'$. Therefore, there exists $\epsilon_1'\in
[\epsilon',\epsilon'/d]$ such that (ii) holds. The proof that (ii)
implies (i) is similar.
\end{proof}

\subsection{Organisation of the proof of Theorem~\ref{theorem.Yclassify2}} The proof is
divided into four parts: we first derive estimates and identities
common to parts (A)--(C) of Theorem~\ref{theorem.Yclassify2}.
Second, we prove \eqref{eq.Ytunstable}, which yields (C). Next, we
obtain the lower bound on the limit superior in
\eqref{eq.Ytboundedh}, which is part of (B). Finally, we find the
upper bound on the limit superior in \eqref{eq.Ytboundedh}, which
completes the proof of the limsup in (B). We also prove
\eqref{eq.Ytto0}, which proves (A).

The proof of the liminf in (B) and the ergodic--type result in part
(B) are not given at this point. Instead, we prove them independently for the
solution of the general equation \eqref{eq.stochlinode}. The results
for $Y$ are then simply corollaries of this general result, with
$A=-I_d$.
\subsection{Preliminary estimates}
Let $V(j):=\int_{t_{j-1}}^{t_j} e^{s-t_j}\sigma(s)\,dB(s)$, $j\geq
1$. Define $V_i(j)=\langle V(j),\mathbf{e}_i\rangle$. Then
\[
V_i(j)=\sum_{l=1}^r \int_{t_{j-1}}^{t_j}
e^{s-t_j}\sigma_{il}(s)\,dB_l(s).
\]
For each fixed $i$, Then $(V_i(j))_{j\geq 1}$ is a sequence of
independently and normally distributed random variables with mean
zero and variance
\[
v_i^2(j-1):=\text{Var}[V_i(j)]=\sum_{l=1}^r \int_{t_{j-1}}^{t_j}
e^{2s-2t_j}\sigma^2_{il}(s)\,ds\leq \sum_{l=1}^r
\int_{t_{j-1}}^{t_j} \sigma_{il}^2(s)\,ds = \theta^2_i(j-1).
\]
Similarly, $v^2_i(j-1)\geq e^{2(t_{j-1}-t_j)}\theta^2_i(j-1)\geq
e^{-2\beta}\theta^2_i(j-1)$, so $v_i(j-1)=0$ if and only if
$\theta_i(j-1)=0$. Also, by \eqref{eq.Yformmult}, we get
\begin{equation} \label{eq.YVconvrept}
Y(t_n)=e^{-t_n}\sum_{j=1}^n \int_{t_{j-1}}^{t_j}
e^s\sigma(s)\,dB(s)=\sum_{j=1}^n e^{-(t_n-t_j)}V(j), \quad n\geq 1.
\end{equation}
This also implies that for $n\geq 1$ we have
\begin{equation} \label{eq.Yautoregt}
Y(t_{n+1})=V(n+1)+\sum_{j=1}^n e^{-(t_{n+1}-t_j)}V(j)=V(n+1)+
e^{-(t_{n+1}-t_n)} Y(t_n).
\end{equation}
Next, as $V_i(j)$ is normally distributed, we have
 $\mathbb{P}[|V_i(j)|>\epsilon]=2(1-\Phi(\epsilon/v_i(j-1))$.
However, as $\Phi$ is increasing, and $e^{-\beta}\theta_i(j-1)\leq
v_i(j-1)\leq \theta_i(j-1)$, we have
$1-\Phi(\epsilon/e^{-\beta}\theta_i(j-1)) \leq
1-\Phi(\epsilon/v_i(j-1))\leq 1-\Phi(\epsilon/\theta_i(j-1))$, so
\begin{equation} \label{eq.Vthetamult}
2\left(1-\Phi(\epsilon/e^{-\beta}\theta_i(j-1))\right) \leq
\mathbb{P}[|V_i(j)|>\epsilon]\leq
2\left(1-\Phi(\epsilon/\theta_i(j-1))\right), \quad j\geq 1.
\end{equation}

Note that $\|V(j)\|_1=\sum_{i=1}^d |V_i(j)|$. Thus, as $\|V(j)\|_1\geq
|V_i(j)|$, we have that $\mathbb{P}[\|V(j)\|_1\geq \epsilon]\geq
\mathbb{P}[|V_i(j)|\geq \epsilon]$ for each $i=1,\ldots,d$.
Therefore
\begin{equation} \label{eq.normVvsVi1}
d\mathbb{P}[\|V(j)\|_1\geq \epsilon]\geq \sum_{i=1}^d
\mathbb{P}[|V_i(j)|\geq \epsilon]. \end{equation} On the other hand,
defining $A_i(j)=\{|V_i(j)|\leq \epsilon/d\}$ and
$B(j)=\{\|V(j)\|_1\leq \epsilon\}$, we see that $\cap_{i=1}^d
A_i(j)\subseteq B(j)$. Then
\begin{equation}   \label{eq.normVvsVi2}
\mathbb{P}[\|V(j)\|_1\geq \epsilon]=\mathbb{P}[\overline{B(j)}]\leq
\mathbb{P}\left[\overline{\cap_{i=1}^d A_i(j)}\right]
=\mathbb{P}\left[\cup_{i=1}^d \overline{A_i(j)}\right] \leq
\sum_{i=1}^d \mathbb{P}\left[|V_i(j)|\geq \epsilon/d\right].
\end{equation}
\subsection{Proof of part (C)}
Suppose $S_{t_\cdot}$ obeys \eqref{eq.thetaunstablet}. Then by
Lemma~\ref{lemma.SS1} we have that $S^1_{t\cdot}(\epsilon)=+\infty$
for every $\epsilon>0$. Therefore by \eqref{eq.Vthetamult},
$\sum_{j=1}^\infty \sum_{i=1}^d
\mathbb{P}[|V_i(j)|>\epsilon]=+\infty$ for every $\epsilon>0$.
Therefore, by \eqref{eq.normVvsVi1} we have $\sum_{j=1}^\infty
\mathbb{P}[\|V(j)\|_1\geq \epsilon]=+\infty$ for all $\epsilon>0$.
Since $(V(j))_{j\geq 1}$ are independent, it follows from the
Borel--Cantelli Lemma that for every $\epsilon>0$
$\limsup_{n\to\infty} \|V(n)\|_1>\epsilon$ a.s. Letting $\epsilon\to
\infty$ through the integers, we have $\limsup_{n\to\infty}
\|V(n)\|_1=+\infty$ a.s. Thus by \eqref{eq.Yautoregt}, we obtain
$\limsup_{n\to\infty} \|Y(t_n)\|_1=+\infty$ a.s., which implies that
$\limsup_{t\to\infty} \|Y(t)\|_1=+\infty$ a.s.
\subsection{Proof of lower bound in part (B)}
Suppose that $S_{t\cdot}$ obeys \eqref{eq.thetaboundedt2}. There
exists an $\epsilon<\epsilon'$ such that $\sum_{j=1}^\infty
\left\{1-\Phi\left(\epsilon/\theta(j)\right)\right\}=+\infty$.
Therefore, by Lemma~\ref{lemma.SS1}, it follows that there exists
$\epsilon_1'>0$ such that for all $\epsilon/e^{-\beta}<\epsilon_1'$
we have
\[
\sum_{j=1}^\infty \sum_{i=1}^d
\left\{1-\Phi\left(\frac{\epsilon}{e^{-\beta}\theta_i(j)}\right)\right\}=+\infty.
\]
By \eqref{eq.Vthetamult} we therefore have
\[
\sum_{j=1}^\infty \sum_{i=1}^d \mathbb{P}[|V_i(j)|>\epsilon] \geq
\sum_{j=1}^\infty 2\left\{1-\Phi\left(\frac{\epsilon
e^{-\beta}}{e^{-\beta}\theta(j-1)}\right)\right\}=+\infty.
\]
Therefore by \eqref{eq.normVvsVi1} we have
\[
\sum_{j=1}^\infty \mathbb{P}[\|V(j)\|_1>\epsilon] =+\infty.
\]
By the independence of $(V(j))$ together with the Borel--Cantelli
Lemma, it follows that $\limsup_{n\to\infty} \|V(n)\|_1\geq \epsilon$
a.s. Letting $\epsilon\uparrow \epsilon_1' e^{-\beta}$ through the
rational numbers gives $\limsup_{n\to\infty} \|V(n)\|_1\geq
\epsilon_1' e^{-\beta}$ on $\Omega_1$, an a.s. event. By
\eqref{eq.Yautoregt}, $V(n+1)=Y(t_{n+1})-e^{-(t_{n+1}-t_n)}Y(t_n)$,
so we have
\[
\epsilon_1' e^{-\beta}\leq \limsup_{n\to\infty} \|V(n,\omega)\|_1\leq
(1+e^{-\alpha}) \limsup_{n\to\infty}\|Y(t_n,\omega)\|_1, \quad
\text{for $\omega\in \Omega_1$}.
\]
Thus
\[
\limsup_{n\to\infty} \|Y(t_n)\|_1\geq \epsilon_1'
e^{-\beta}/(1+e^{-\alpha}),\quad\text{a.s.}, \] which implies
$\limsup_{t\to\infty} \|Y(t)\|_1\geq  \epsilon_1'
e^{-\beta}/(1+e^{-\alpha})=:c_1$, a.s.
\subsection{Proof of upper bounds in parts (A) and (B)}
Suppose that
\begin{equation} \label{eq.Vprobbound}
\sum_{j=1}^\infty
\left\{1-\Phi\left(\frac{\epsilon}{\theta(j)}\right)\right\}<+\infty.
\end{equation}
In part (A), \eqref{eq.Vprobbound} holds for all $\epsilon>0$, while
in part (B) it holds for all $\epsilon>\epsilon'$. By
\eqref{eq.Vprobbound} and \eqref{eq.phi1} we have
\begin{equation*}
\sum_{j=1}^\infty \sum_{i=1}^d
\left\{1-\Phi\left(\frac{\epsilon}{\theta_i(j)}\right)\right\}<+\infty,
\end{equation*}
and hence by \eqref{eq.Vthetamult} we have
\begin{equation*}
\sum_{j=1}^\infty \sum_{i=1}^d \mathbb{P}[|V_i(j)|\geq
\epsilon]<+\infty.
\end{equation*}
By the Borel--Cantelli lemma, it follows that
$\limsup_{n\to\infty}|V_i(n)|\leq \epsilon$ a.s. Now from
\eqref{eq.YVconvrept}, we have that
\[Y_i(t_n)=\sum_{j=1}^n e^{-(t_n-t_j)}V_i(j),
\]
so therefore, as $t_n-t_j\geq \alpha(n-j)$ for $j=1,\ldots,n$, we
have that
\[
|Y_i(t_n)|\leq \sum_{j=1}^n e^{-(t_n-t_j)}|V_i(j)|\leq \sum_{j=1}^n
e^{-\alpha(n-j)}|V_i(j)|,
\]
so
\begin{equation} \label{eq.Yiest} \limsup_{n\to\infty}
|Y_i(t_n)|\leq \epsilon \sum_{j=0}^\infty e^{-\alpha j} = \epsilon
\frac{1}{1-e^{-\alpha}}, \quad \text{a.s.}
\end{equation}


Next let $t\in [t_n,t_{n+1})$. Therefore, from \eqref{def.Y} we have
\[
Y_i(t)=Y_i(t_n)e^{-(t-t_n)}+ \sum_{l=1}^r e^{-t}\int_{t_n}^t
e^s\sigma_{il}(s)\,dB_l(s), \quad t\in [t_n,t_{n+1}).
\]
Therefore
\begin{align}
\lefteqn{\max_{t\in [t_n,t_{n+1}]} |Y_i(t)|}\nonumber\\
\label{eq.YtYnZn}
 &\leq |Y_i(t_n)|+ \max_{t\in [t_n,t_{n+1}]} e^{-t}
\left|\sum_{l=1}^r \int_{t_n}^t
e^s\sigma_{il}(s)\,dB_l(s)\right|\leq |Y_i(t_n)|+ Z_i(n),
\end{align}
where
\[
Z_i(n):=e^{-t_n}\max_{t\in [t_n,t_{n+1}]}  \left|\sum_{l=1}^r
\int_{t_n}^t e^s\sigma_{il}(s)\,dB_l(s)\right|, \quad n\geq 1.
\]

Fix $n\in \mathbb{N}$. Now
\[
\mathbb{P}[Z_i(n)>\epsilon ]=\mathbb{P}\left[\max_{t\in
[t_n,t_{n+1}]} \left|\sum_{l=1}^r \int_{t_n}^t
e^s\sigma_{il}(s)\,dB_l(s)\right|>\epsilon  e^{t_n}\right]
\]
Define $\tau_i(t):=\sum_{l=1}^r \int_{t_n}^t
e^{2s}\sigma_{il}^2(s)\,ds$ for $t\in[n,n+1]$. Consider
\[
C_{in}(t)=\sum_{l=1}^r \int_{t_n}^t e^s\sigma_{il}(s)\,dB_l(s),
\quad t\in [t_n,t_{n+1}].
\]
Then $C_{in}=\{C_{in}(t)\,:\,t_n\leq t\leq t_{n+1}\}$ is a
continuous martingale with $\langle C_{in}\rangle(t)=\tau_i(t)$ for
$t\in [t_n,t_{n+1}]$. Therefore, by the martingale time change
theorem ~\cite[Theorem V.1.6]{RevYor}, there exists a standard
Brownian motion $B^\ast_{in}$ such that
$C_{in}(t)=B^\ast_{in}(\tau_i(t))$ for $t\in [t_n,t_{n+1}]$, and so
we have
\begin{align*}
\mathbb{P}[Z_i(n)>\epsilon ]
&=\mathbb{P}\left[\max_{t\in [t_n,t_{n+1}]}
\left|B^\ast_{in}\left(\sum_{l=1}^r \int_{t_n}^t e^{2s}\sigma_{il}^2(s)\,ds\right)\right|>\epsilon  e^{t_n}\right]\\
&=\mathbb{P}\left[\max_{u\in [0,\tau_i(n+1)]} \left|B^\ast_{in}(u)\right|>\epsilon   e^{t_n}\right]\\
&=\mathbb{P}\left[\left\{\max_{u\in [0,\tau_i(n+1)]}
B^\ast_{in}(u)>\epsilon  e^{t_n}\right\}
\cup \left\{\max_{u\in [0,\tau_i(n+1)]} -B^\ast_{in}(u)>\epsilon e^{t_n}\right\}\right]\\
&\leq \mathbb{P}\left[\max_{u\in [0,\tau_i(n+1)]}
B^\ast_{in}(u)>\epsilon   e^{t_n}\right]
+\mathbb{P}\left[\max_{u\in [0,\tau_i(n+1)]} -B^\ast_{in}(u)>\epsilon   e^{t_n}\right]\\
&=\mathbb{P}\left[|B^\ast_{in}(\tau_i(n+1))|>\epsilon e^{t_n}\right]
+\mathbb{P}\left[ |B^{\ast\ast}_{in}(\tau_i(n+1))|>\epsilon
e^{t_n}\right],
\end{align*}
where $B^{\ast\ast}_{in}=-B^\ast_{in}$ is a standard Brownian
motion. Recall that if $W$ is a standard Brownian motion that
$\max_{s\in [0,t]} W(s)$ has the same distribution as $|W(t)|$.
Therefore, as $B^\ast_{in}(\tau_i(n+1))$ is normally distributed
with zero mean we have
\begin{align*}
\mathbb{P}[Z_i(n)>\epsilon] &\leq
2\mathbb{P}\left[|B^\ast_{in}(\tau_i(n+1))|>\epsilon e^{t_n}\right]
=4\mathbb{P}\left[B^\ast_{in}(\tau_i(n+1))>\epsilon e^{t_n}\right]\\
&=4\left(1-\Phi\left(\frac{\epsilon
e^{t_n}}{\sqrt{\tau_i(n+1)}}\right)\right)
=4\left(1-\Phi\left(\frac{\epsilon}{\sqrt{e^{-2t_n}\tau_i(n+1)}}\right)\right).
\end{align*}
If we interpret $\Phi(\infty)=1$, this formula holds valid in the
case when $\tau_i(n+1)=0$, because in this situation $Z_i(n)=0$ a.s.
Now
\begin{multline*}
e^{-2t_n}\tau_i(n+1)=e^{-2t_n}\sum_{l=1}^r \int_{t_n}^{t_{n+1}}
e^{2s}\sigma^2_{il}(s)\,ds\\
\leq e^{2(t_{n+1}-t_n)}\int_n^{n+1} \sigma_{il}^2(s)\,ds \leq
e^{2\beta} \theta_i^2(n). \end{multline*} Since $\Phi$ is
increasing, we have
\[
\mathbb{P}[Z_i(n)>\epsilon]\leq 4\left(
1-\Phi\left(\frac{\epsilon}{\sqrt{e^{-2t_n}\tau_i(n+1)}}\right)\right)\leq
4\left( 1-
\Phi\left(\frac{\epsilon}{e^{\beta}\theta_i(n)}\right)\right).
\]
Therefore we have
\begin{equation} \label{eq.Ziest}
\mathbb{P}[Z_i(n)>\epsilon e^\beta]\leq 4\left( 1-
\Phi\left(\frac{\epsilon}{\theta_i(n)}\right)\right), \quad n\geq 0.
\end{equation}
Hence
\[
\sum_{n=1}^\infty \mathbb{P}[Z_i(n)>\epsilon e^\beta]<+\infty,
\]
so by the Borel--Cantelli lemma, we have that
\begin{equation} \label{eq.Zilimsup}
\limsup_{n\to\infty} Z_i(n) \leq \epsilon e^\beta, \quad\text{a.s.}
\end{equation}
Therefore by \eqref{eq.YtYnZn}, \eqref{eq.Yiest} and
\eqref{eq.Zilimsup}, we have that
\begin{equation*}
\limsup_{t\to\infty} |Y_i(t)| \leq (1/(1-e^{-\alpha}) +
e^\beta)\epsilon, \quad\text{a.s.}
\end{equation*}
and so
\begin{equation}\label{eq.Yboundedepsilonmult}
\limsup_{t\to\infty} \|Y(t)\|_1 \leq d(1/(1-e^{-\alpha}) +
e^\beta)\epsilon, \quad\text{a.s.}
\end{equation}

If \eqref{eq.thetastableh} holds, \eqref{eq.Yboundedepsilonmult}
implies that $Y(t)\to 0$ as $t\to\infty$ a.s.

If the first part of \eqref{eq.thetaboundedh} holds, then
\eqref{eq.Yboundedepsilonmult} holds for every $\epsilon>\epsilon'$.
Thus, letting $\epsilon\downarrow \epsilon'$ through the rational
numbers we have $\limsup_{t\to\infty}\|Y(t)\|_1 \leq
d(1/(1-e^{-\alpha})+e^\beta)\epsilon'=:c_2$ a.s., proving
\eqref{eq.Ytboundedh}.

\section{Proof of Theorem~\ref{theorem.Icondn}}
We start by proving a preliminary lemma.
\begin{lemma} \label{lemma.xintsum}
Suppose $x\in C([0,\infty);[0,\infty))$.
\begin{itemize}
\item[(i)] If $\int_0^\infty x(t)\,dt=+\infty$, then for every $h>0$
there exists a sequence $(t_n)_{n\geq 0}$ obeying
\[
t_0=0, \quad h\leq t_{n+1}-t_n\leq 3h, \quad n\geq 0
\]
such that
\begin{equation} \label{eq.xtnnotsumm}
\sum_{n=0}^\infty x(t_n)=+\infty
\end{equation}
\item[(ii)] If $\int_0^\infty x(t)\,dt<+\infty$, then for every $h>0$
there exists a sequence $(t_n)_{n\geq 0}$ obeying
\[
t_0=0, \quad h\leq t_{n+1}-t_n\leq 3h, \quad n\geq 0
\]
such that
\begin{equation}  \label{eq:eq19}
\sum_{n=0}^\infty x(t_n)<+\infty
\end{equation}
\end{itemize}
\end{lemma}
\begin{proof}
We start by proving part (i). Let $s_0=0$ and define for $n\geq 1$
\begin{equation} \label{eq.suminfinitetn}
s_n=\inf\{t\in[nh,(n+1)h]\,:\, x(t)=\max_{s\in [nh,(n+1)h]} x(s)\}.
\end{equation}
Clearly $s_n\in [nh,(n+1)h]$. Thus
\[
+\infty = \int_0^\infty x(t)\,dt = \int_0^h x(s)\,ds
+\sum_{n=1}^\infty \int_{nh}^{(n+1)h} x(s)\,ds \leq h \max_{s\in
[0,h]} x(s) + \sum_{n=1}^\infty h x(s_n).
\]
Therefore we have
\[
\sum_{n=1}^\infty x(s_{2n})+\sum_{n=0}^\infty x(s_{2n+1})=+\infty.
\]
Hence we have that either (I) $\sum_{n=1}^\infty x(s_{2n})=+\infty$
or (II) $\sum_{n=0}^\infty x(s_{2n+1})=+\infty$.

If case (I) holds, let $t_n=s_{2n}$ for $n\geq 0$. Then $t_0=0$ and
$(t_n)_{n\geq 0}$ obeys \eqref{eq.xtnnotsumm}. 
Note that $t_1-t_0=t_1=s_2\in [2h,3h]$. For $n\geq 1$, we have
$t_{n+1}-t_n=s_{2n+2}-s_{2n}$. Hence $t_{n+1}-t_n\leq
(2n+3)h-2nh=3h$. Also $t_{n+1}-t_n\geq (2n+2)h-(2n+1)h=h$. Therefore
$t_n$ obeys all the required properties.

If case (II) holds, let $t_n=s_{2n-1}$ for $n\geq 1$ and $t_0=0$.
Then $t_0=0$ and $(t_n)_{n\geq 0}$ obeys \eqref{eq.xtnnotsumm}.
Note that $t_1-t_0=t_1=s_1\in [h,2h]$. Therefore $h\leq t_1-t_0\leq
2h<3h$. For $n\geq 1$, we have $t_{n+1}-t_n=s_{2n+1}-s_{2n-1}$.
Hence $t_{n+1}-t_n\leq (2n+2)h-(2n-1)h=3h$. Also $t_{n+1}-t_n\geq
(2n+1)h-(2n-1+1)h=h$. Therefore $t_n$ obeys all the required
properties.

We now turn to the proof of part (ii).
Construct $(t_n)_{n=0}^{\infty}$ recursively as follows: let
$t_0=0$, and for $n \in \mathbb{N}$ \begin{equation}
\label{eq.sumfinitetn} t_{n+1}=\inf\{t \in [t_n+h,t_{n}+2h]\::\:
x(t) = \min_{t_n+h \leq s \leq t_n+2h} x(s)\}.
\end{equation}
The existence of such a sequence can be proved by induction on $n$,
taking note that $x$ is continuous on the compact interval
$[t_n+h,t_n+2h]$, and hence attains its minimum. By construction, we
have
\begin{equation}   \label{eq:eq4.1}
t_{n+1}-t_n \geq h>0,
\end{equation}
and $t_{n+1}-t_{n} \leq 2h$. 
To prove \eqref{eq:eq19}, note that $x(t_{n+1}) \leq x(t)$ for
$t_{n}+h \leq t \leq t_{n}+2h$, so by integrating both sides of this
inequality over $[t_{n}+h,t_{n}+2h]$, using the non-negativity of
$x(\cdot)$ and $t_{n}+2h\leq t_{n+1}-h+2h=t_{n+1}+h$ (which follows
from \eqref{eq:eq4.1}), we get
\[
h x(t_{n+1}) \leq \int_{t_{n}+h}^{t_{n}+2h} x(t)\,dt \leq
\int_{t_{n}+h}^{t_{n+1}+h} x(t)\,dt.
\]
Summing both sides of this inequality establishes \eqref{eq:eq19}.
\end{proof}

\begin{lemma} \label{lemma.Iepssum}
Suppose that $I$ is defined by \eqref{def.Ieps}.
\begin{itemize}
\item[(i)]
Suppose that $I(\epsilon)=+\infty$. Then there exists $(t_n)_{n\geq
0}$ independent of $\epsilon>0$ such that
\[
t_0=0, \quad 0<h \leq t_{n+1}-t_n\leq 3h<+\infty, \quad n\geq 0,
\]
and
\[
\sum_{n=0}^\infty\left\{
1-\Phi\left(\frac{\epsilon}{\sqrt{\int_{t_n}^{t_{n+1}}
\|\sigma(s)\|^2_F\,ds}} \right)\right\}=+\infty.
\]
\item[(ii)]
Suppose that $I(\epsilon)<+\infty$. Then there exists $(t_n)_{n\geq
0}$ independent of $\epsilon>0$ such that
\[
t_0=0, \quad 0<h \leq t_{n+1}-t_n\leq 3h<+\infty, \quad n\geq 0,
\]
and
\[
\sum_{n=0}^\infty\left\{
1-\Phi\left(\frac{\epsilon}{\sqrt{\int_{t_n}^{t_{n+1}}
\|\sigma(s)\|^2_F\,ds}} \right)\right\}<+\infty.
\]
\end{itemize}
\end{lemma}
\begin{proof}
Define \begin{equation}\label{def.zeta}
\zeta^2(t)=\int_t^{t+c}\|\sigma(s)\|^2_F\,ds, \quad t\geq 0,
\end{equation}
and $\phi_\epsilon(x)=xe^{-\epsilon^2/(2x^2)}\chi_{(0,\infty)}(x)$
for $x\geq 0$. Therefore for $x\geq 0$ we have
\[
\frac{1}{\epsilon} \phi_\epsilon(x)=\frac{1}{\epsilon}
xe^{-\epsilon^2/(2x^2)}\chi_{(0,\infty)}(x/\epsilon)=\phi_1(x/\epsilon).
\]
Then
\[
I(\epsilon)/\epsilon=\int_0^\infty
\phi_\epsilon\left(\zeta(t)\right)/\epsilon\,dt=\int_0^\infty
\phi_1\left(\zeta(t)/\epsilon\right)\,dt.
\]
Let $x_\epsilon(t)=\phi_1(\zeta(t)/\epsilon)$ for $t\geq 0$. Clearly
$x$ is a non--negative function on $[0,\infty)$, and as $\lim_{x\to
0^+} \phi_1(x)=0=\phi_1(0)$, we have that $\phi_1$ is continuous and
increasing on $[0,\infty)$. Hence $x_\epsilon$ is continuous on
$[0,\infty)$. Note therefore that
$I(\epsilon)/\epsilon=\int_0^\infty x_\epsilon(t)\,dt$.

We are now in a position to prove part (ii). Suppose that
$I(\epsilon)<+\infty$. Let $0<h\leq c/3$. Then by
Lemma~\ref{lemma.xintsum} part (ii) there exists $(t_n)_{n\geq 0}$
such that $h\leq t_{n+1}-t_n\leq 3h$ and $\sum_{n=0}^\infty
\phi_\epsilon(\zeta(t_n))<+\infty$. Recall that $t_n$ are defined by
\eqref{eq.sumfinitetn} i.e., $t_0=0$, and for $n \in \mathbb{N}$ we
have
\[
t_{n+1}=\inf\{t \in [t_n+h,t_{n}+2h]\::\: x_\epsilon(t) =
\min_{t_n+h \leq s \leq t_n+2h} x_\epsilon(s)\}.
\]
Since $x_\epsilon(t)=\phi_1(\zeta(t)/\epsilon)$ and $\phi_1$ is
increasing, it follows that
\[
t_{n+1}=\inf\{t \in [t_n+h,t_{n}+2h]\::\: \zeta(t) = \min_{t_n+h
\leq s \leq t_n+2h} \zeta(s)\},
\]
and since $\zeta$ is independent of $\epsilon$, it follows that
$(t_n)$ is independent of $\epsilon$.

 $\sum_{n=0}^\infty
\phi_\epsilon(\zeta(t_n))<+\infty$ is therefore equivalent to
\[
\sum_{n=0}^\infty
\zeta(t_n)\exp\left(-\frac{\epsilon^2}{2}\frac{1}{\zeta(t_n)^2}\right)<+\infty.
\]
This implies that $\zeta(t_n)\to 0$ as $n\to\infty$, and by \eqref{eq.millsasy} we have that
\[ \lim_{n\to\infty}
\frac{1-\Phi(\epsilon/\zeta(t_n))}{\frac{\zeta(t_n)}{\epsilon}
\exp\left(-\frac{\epsilon^2}{2}\frac{1}{\zeta^2(t_n)}\right)}=\frac{1}{\sqrt{2\pi}}.
\]
Hence we have
\begin{equation} \label{eq.phizetavsPhi1}
\sum_{n=0}^\infty
\left\{1-\Phi\left(\frac{\epsilon}{\sqrt{\int_{t_n}^{t_n+c}
\|\sigma(s)\|^2_F\,ds} } \right) \right\} = \sum_{n=0}^\infty
\left\{1-\Phi\left(\frac{\epsilon}{\zeta(t_n)} \right)
\right\}<+\infty.
\end{equation}
Since $t_{n+1}\leq t_n+3h$, and $3h\leq c$, we have
\[
\int_{t_n}^{t_n+c} \|\sigma(s)\|^2_F\,ds\geq \int_{t_n}^{t_n+3h}
\|\sigma(s)\|^2_F\,ds\geq \int_{t_n}^{t_{n+1}}
\|\sigma(s)\|^2_F\,ds.
\]
Since $\Phi$ is increasing, we have
\[
1-\Phi\left(\frac{\epsilon}{\sqrt{\int_{t_n}^{t_n+c}
\|\sigma(s)\|^2_F\,ds}}\right)\geq
1-\Phi\left(\frac{\epsilon}{\sqrt{\int_{t_n}^{t_{n+1}}
\|\sigma(s)\|^2_F\,ds}}\right).
\]
By \eqref{eq.phizetavsPhi1} we have \begin{multline*}
\sum_{n=0}^\infty \left\{
1-\Phi\left(\frac{\epsilon}{\sqrt{\int_{t_n}^{t_{n+1}}
\|\sigma(s)\|^2_F\,ds}}\right) \right\} \\
\leq \sum_{n=0}^\infty
\left\{1-\Phi\left(\frac{\epsilon}{\sqrt{\int_{t_n}^{t_n+c}
\|\sigma(s)\|^2_F\,ds} } \right) \right\}<+\infty,
\end{multline*}
which proves part (ii).

We are now in a position to prove part (i). Suppose that
$I(\epsilon)=+\infty$. Let $h\in [c,\infty)$. Then by part (i) of
Lemma~\ref{lemma.xintsum} there exists $(t_n)_{n\geq 0}$ such that
$h\leq t_{n+1}-t_n\leq 3h$ and $\sum_{n=0}^\infty
\phi_\epsilon(\zeta(t_n))=+\infty$. We now wish to show that the
$(t_n)$ are independent of $\epsilon>0$. Since they depend directly
on the sequence $(s_n)$ defined by \eqref{eq.suminfinitetn}, we must
simply show that the sequence $(s_n)$ is independent of $\epsilon$.
By \eqref{eq.suminfinitetn} we have
\begin{equation*} 
s_n=\inf\{t\in[nh,(n+1)h]\,:\, x_\epsilon(t)=\max_{s\in [nh,(n+1)h]}
x_\epsilon(s)\}.
\end{equation*}
Since $x_\epsilon(t)=\phi_1(\zeta(t)/\epsilon)$ and $\phi_1$ is
increasing, it follows that
\[
s_n=\inf\{t\in[nh,(n+1)h]\,:\, \zeta(t)=\max_{s\in [nh,(n+1)h]}
\zeta(s)\},
\]
and since $\zeta$ is independent of $\epsilon$, so are $(s_n)$ and
therefore $(t_n)$.

Next, $\sum_{n=0}^\infty \phi_\epsilon(\zeta(t_n))=+\infty$ is
equivalent to
\[
\sum_{n=0}^\infty
\zeta(t_n)\exp\left(-\frac{\epsilon^2}{2}\frac{1}{\zeta(t_n)^2}\right)=+\infty.
\]
Suppose that
\[
\sum_{n=0}^\infty \left\{1-\Phi\left(\frac{\epsilon}{\zeta(t_n)}
\right) \right\} <+\infty.
\]
Then  $\zeta(t_n)\to 0$ as $n\to\infty$, and by \eqref{eq.millsasy} we have
\[
\lim_{n\to\infty}
\frac{1-\Phi(\epsilon/\zeta(t_n))}{\frac{\zeta(t_n)}{\epsilon}
\exp\left(-\frac{\epsilon^2}{2}\frac{1}{\zeta^2(t_n)}\right)}=\frac{1}{\sqrt{2\pi}}.
\]
Hence we have that
\[
\sum_{n=0}^\infty
\zeta(t_n)\exp\left(-\frac{\epsilon^2}{2}\frac{1}{\zeta(t_n)^2}\right)<+\infty,
\]
a contradiction. Therefore we have
\begin{equation} \label{eq.phizetavsPhi2}
\sum_{n=0}^\infty
\left\{1-\Phi\left(\frac{\epsilon}{\sqrt{\int_{t_n}^{t_n+c}
\|\sigma(s)\|^2_F\,ds} } \right) \right\} =\sum_{n=0}^\infty
\left\{1-\Phi\left(\frac{\epsilon}{\zeta(t_n)} \right) \right\}
=+\infty.
\end{equation}
Next, as $c\leq h$ and $t_{n+1}\geq t_n+h$ we have
\[
\int_{t_n}^{t_n+c} \|\sigma(s)\|^2_F\,ds \leq \int_{t_n}^{t_n+h}
\|\sigma(s)\|^2_F\,ds \leq \int_{t_n}^{t_{n+1}}
\|\sigma(s)\|^2_F\,ds.
\]
Since $\Phi$ is increasing, we have
\[
1-\Phi\left(\frac{\epsilon}{\sqrt{\int_{t_n}^{t_n+c}
\|\sigma(s)\|^2_F\,ds}}\right)\leq
1-\Phi\left(\frac{\epsilon}{\sqrt{\int_{t_n}^{t_{n+1}}
\|\sigma(s)\|^2_F\,ds}}\right).
\]
By \eqref{eq.phizetavsPhi2} we have
\begin{multline*}
\sum_{n=0}^\infty \left\{
1-\Phi\left(\frac{\epsilon}{\sqrt{\int_{t_n}^{t_{n+1}}
\|\sigma(s)\|^2_F\,ds}}\right) \right\} \\
\geq \sum_{n=0}^\infty
\left\{1-\Phi\left(\frac{\epsilon}{\sqrt{\int_{t_n}^{t_n+c}
\|\sigma(s)\|^2_F\,ds} } \right) \right\}=+\infty,
\end{multline*}
which proves part (i).
\end{proof}

\begin{proof}[Proof of Theorem~\ref{theorem.Icondn}]
To prove part (A), we have by hypothesis that $I(\epsilon)<+\infty$
for all $\epsilon>0$. Then, by Lemma~\ref{lemma.Iepssum} part (ii),
for every $h\leq c/3$, there exists $(t_n)_{n\geq 0}$ independent of
$\epsilon$ for which $h\leq t_{n+1}-t_n\leq 3h$ and
\[
\sum_{n=0}^\infty\left\{
1-\Phi\left(\frac{\epsilon}{\sqrt{\int_{t_n}^{t_{n+1}}
\|\sigma(s)\|^2_F\,ds}} \right)\right\}<+\infty.
\]
Therefore by Theorem~\ref{theorem.Yclassify2} part (A), it follows
that $Y(t)\to 0$ as $t\to\infty$ a.s.

To prove part (C), we have by hypothesis that $I(\epsilon)=+\infty$
for all $\epsilon>0$. Then, by Lemma~\ref{lemma.Iepssum} part (i),
for every $h\geq c$, there exists $(t_n)_{n\geq 0}$ independent of
$\epsilon$ for which $h\leq t_{n+1}-t_n\leq 3h$ and
\[
\sum_{n=0}^\infty\left\{
1-\Phi\left(\frac{\epsilon}{\sqrt{\int_{t_n}^{t_{n+1}}
\|\sigma(s)\|^2_F\,ds}} \right)\right\}=+\infty.
\]
Therefore by Theorem~\ref{theorem.Yclassify2} part (C), it follows
that $\limsup_{t\to\infty}\|Y(t)\|=+\infty$ a.s.

To prove part (B), we have by hypothesis that $I(\epsilon)<+\infty$
for all $\epsilon>\epsilon'$. Then, by Lemma~\ref{lemma.Iepssum}
part (ii), for every $h\leq c/3$, there exists $(t_n)_{n\geq 0}$
independent of $\epsilon$ for which $h\leq t_{n+1}-t_n\leq 3h$ and
\[
\sum_{n=0}^\infty\left\{
1-\Phi\left(\frac{\epsilon}{\sqrt{\int_{t_n}^{t_{n+1}}
\|\sigma(s)\|^2_F\,ds}} \right)\right\}<+\infty.
\]
Therefore by Theorem~\ref{theorem.Yclassify2} part (B), it follows
that $\limsup_{t\to\infty}\|Y(t)\|\leq c_2$ a.s.

On the other hand, we have by hypothesis that $I(\epsilon)=+\infty$
for all $\epsilon<\epsilon'$. Then, by Lemma~\ref{lemma.Iepssum}
part (i), for every $h\geq c$, there exists $(\tau_n)_{n\geq 0}$
independent of $\epsilon$ for which $h\leq \tau_{n+1}-\tau_n\leq 3h$
and
\[
\sum_{n=0}^\infty\left\{
1-\Phi\left(\frac{\epsilon}{\sqrt{\int_{\tau_n}^{\tau_{n+1}}
\|\sigma(s)\|^2_F\,ds}} \right)\right\}=+\infty.
\]
Therefore by Theorem~\ref{theorem.Yclassify2} part (B), it follows
that $\limsup_{t\to\infty}\|Y(t)\|\geq c_1$ a.s.
\end{proof}

\section{Proofs of Theorem~\ref{theorem.XYasy} and \ref{theorem.nonstabilise}}
\subsection{Proof of Theorem~\ref{theorem.XYasy}}
Let $z(t,\omega)=X(t,\omega)-Y(t,\omega)$ for $t\geq 0$. Then
$z(0)=\xi$ and
\[
z'(t,\omega)=AX(t,\omega)+Y(t,\omega)=Az(t,\omega)+g(t,\omega),
\quad t\geq 0
\]
where
\begin{equation}\label{def.g}
g(t,\omega)=AY(t,\omega)+Y(t,\omega), \quad t\geq 0.
\end{equation}
 Let $\Psi$ be the
unique continuous $d\times d$--valued matrix solution of
\[
\Psi'(t)=A\Psi(t), \quad t\geq 0; \quad \Psi(0)=I_d.
\]
Since all eigenvalues of $A$ have negative real parts, there exist
$K>0$ and $\lambda>0$ such that
\begin{equation} \label{eq.Psiexpest}
\|\Psi(t)\|_2\leq Ke^{-\lambda t}, \quad t\geq 0.
\end{equation}
Now by variation of constants, $z$ is given by
\begin{equation} \label{eq.zvarp}
z(t,\omega)=\Psi(t)\xi+\int_0^t \Psi(t-s)g(s,\omega)\,ds, \quad
t\geq 0.
\end{equation}

To prove statement (A), 
suppose that $Y(t,\omega)\to 0$ as $t\to\infty$ for all $\omega\in
\Omega^\ast$ where $\Omega^\ast$ is an a.s. event. We show now that
$X(t,\xi,\omega)\to 0$ as $t\to\infty$ for every
$\xi\in\mathbb{R}^d$ and every $\omega\in \Omega^\ast$, which would
prove statement (A). Since $Y(t,\omega)\to 0$ as $t\to\infty$ we
have $g(t,\omega)\to 0$ as $t\to\infty$. Therefore by
\eqref{eq.zvarp}, we have $z(t,\omega)\to 0$ as $t\to\infty$. Since
$Y(t,\omega)\to 0$ as $t\to\infty$ and $\Psi(t)\to 0$ as
$t\to\infty$, it follows that $X(t,\omega)\to 0$ as $t\to\infty$.

To prove the upper bound in part (B), note that because there is a deterministic
$c_2>0$ such that $\limsup_{t\to\infty} \|Y(t)\|_2\leq c_2$ a.s., we
have
\[
\limsup_{t\to\infty} \|g(t)\|_2\leq \|I+A\|_2c_2, \quad \text{a.s.}
\]
Using this estimate, the fact that $\Psi(t)\to 0$ as $t\to\infty$,
and \eqref{eq.zvarp} we get
\[
\limsup_{t\to\infty} \|z(t)\|_2\leq \int_0^\infty \|\Psi(s)\|_2\,ds \cdot
\|I+A\|_2c_2=:c_4, \quad \text{a.s.}
\]
Hence we have $\limsup_{t\to\infty} \|X(t)\|_2\leq c_2+c_4=:c_5$ a.s.,
which proves the upper estimate in (B).

To prove the lower bound in part (B), notice that by rewriting
\eqref{eq.stochlinode} in the form
\[
dX(t)=\left(-X(t)+\left\{AX(t)+X(t)\right\}\right)\,dt + \sigma(t)\,dB(t),
\]
and by using stochastic integration by parts and deterministic variation of constants, we arrive at
\[
X(t)=\xi e^{-t} + \int_{0}^t e^{-(t-s)}(I+A)X(s)\,ds + Y(t), \quad t\geq 0.
\]
Therefore, we have that
\begin{equation} \label{eq.YX}
Y(t)=X(t)-\xi e^{-t} - \int_{0}^t e^{-(t-s)}(I+A)X(s)\,ds,\quad t\geq 0.
\end{equation}
Suppose now that $\Omega_1=\{\omega:\limsup_{t\to\infty} \|Y(t,\omega)\|_2\geq c_1\}$, where it is already known that
$\Omega_1$ is an a.s. event.
Then for $\omega\in \Omega_1$, we have
\begin{multline*}
c_1\leq \limsup_{t\to\infty}\|Y(t,\omega)\|_2\\
\leq \limsup_{t\to\infty}\|X(t,\omega)\|_2 +
\|I+A\|_2\limsup_{t\to\infty}\int_{0}^t e^{-(t-s)}\|X(s,\omega)\|_2\,ds.
\end{multline*}
Therefore we arrive at
\[
c_1\leq (1+\|I+A\|_2)\limsup_{t\to\infty}\|X(t,\omega)\|_2,
\]
for each $\omega\in\Omega_1$, and so
\[
\limsup_{t\to\infty}\|X(t)\|_2 \geq c_3:=\frac{c_1}{1+\|I+A\|_2}, \quad\text{a.s.},
\]
as required.


To prove statement (C), we start by noting by hypothesis that the event
$\Omega_2=\{\omega:\limsup_{t\to\infty} \|Y(t,\omega)\|_2=+\infty\}$ is almost sure. Now suppose that there is an event $C=\{\omega: \limsup_{t\to\infty} \|X(t,\omega)\|<+\infty\}\cap \Omega_2$
such that $\mathbb{P}[C]>0$. Taking norms on both sides of \eqref{eq.YX} for $\omega\in C$ yields
\[
\|Y(t,\omega)\|_2\leq \|X(t,\omega)\|_2 +\|\xi\|_2 e^{-t} + \|I+A\|_2 \int_{0}^t e^{-(t-s)} \|X(s,\omega)\|_2\,ds.
\]
Define for $\omega\in C$ the finite $c(\omega):=\limsup_{t\to\infty} \|X(t,\omega)\|_2$. Then
\[
+\infty=
\limsup_{t\to\infty}
\|Y(t,\omega)\|_2\leq c(\omega) + \|I+A\|_2 c(\omega)<+\infty,
\]
a contradiction. Therefore, we must have that $\limsup_{t\to\infty}\|X(t)\|=+\infty$ a.s. as required.

\subsection{Proof of Theorem~\ref{theorem.nonstabilise}}
Theorem~\ref{theorem.IcondnX} shows that (A) implies (C), and (C)
clearly implies (B). It remains to prove that (B) implies (A).
Define $\xi_0=0$ and for $i=1,\ldots,d$ set
$\zeta_i=\xi_i-\xi_{i-1}$.
Next, for $\omega\in C$, define
$V_i(t,\omega)=X(t,\xi_i,\omega)-X(t,\xi_{i-1},\omega)$ for
$i=1,\ldots,d$. Therefore by hypothesis we have that
$V_i(t,\omega)\to 0$ as $t\to\infty$. Moreover, we see that $V_i$
obeys the differential equation
\[
V_i'(t,\omega)=AV_i(t,\omega), \quad t\geq 0, \quad
V_i(0,\omega)=\xi_i-\xi_{i-1}=\zeta_i.
\]
If $\Psi\in \mathbb{R}^{d\times d}$ is the principal matrix solution
given by $\Psi'(t)=A\Psi(t)$ with $\Psi(0)=I_d$, then
$V_i(t,\omega)=\Psi(t)\zeta_i$. Therefore we have that
$\Psi(t)\zeta_i\to 0$ as $t\to\infty$ for each $i=1,\ldots,d$. Since
$(\zeta_i)_{i=1}^d$ are linearly independent, we have that
$\Psi(t)\to 0$ as $t\to\infty$. Hence it follows that all the
eigenvalues of $A$ have negative real parts.

Let $Y$ be the solution of \eqref{def.Y}. Writing $X$ as
\[
dX(t)=\left(-X(t)+\{X(t)+AX(t)\}\right)\,dt + \sigma(t)\,dB(t),
\]
by variation of constants, we see that
\[
X(t)=X(0)e^{-t} + \int_{0}^t e^{-(t-s)}\{X(s)+AX(s)\}\,ds + Y(t),
\quad t\geq 0.
\]
Therefore, we see that $Y(t,\omega)\to 0$ as $t\to\infty$ for each
$\omega\in C$. Since $C$ is an event of positive probability, we see
from Theorem~\ref{theorem.Icondn} that $Y(t)\to 0$ as $t\to\infty$
a.s., and that therefore $I(\epsilon)$ is finite for all
$\epsilon>0$. We have therefore shown that (B) implies both
conditions in (A), as required.

\section{Proof of \eqref{eq.liminfYaveY0} in part (B) of Theorems~\ref{theorem.Yclassify3} and \ref{theorem.Icondn}
and of \eqref{eq.IcondnXbound} in part (B) of Theorems~\ref{theorem.liniffsigma},~\ref{theorem.IcondnX}}
We note first that the proof of \eqref{eq.liminfYaveY0}  in part (B) of Theorem~\ref{theorem.Yclassify3} is a direct corollary
of part (B) in Theorem~\ref{theorem.liniffsigma}, where $A=-I_d$. Similarly, the proof of \eqref{eq.liminfYaveY0}  in part (B) of Theorem~\ref{theorem.Icondn} is a corollary of part (B) in Theorem~\ref{theorem.IcondnX}.

To prove \eqref{eq.IcondnXbound} in  Theorem~\ref{theorem.liniffsigma}, it suffices to show that $\|X\|$ being bounded a.s.
and $S_1(\epsilon)<+\infty$ for some $\epsilon>0$  implies  \eqref{eq.IcondnXbound}; on the other hand, to prove
 \eqref{eq.IcondnXbound} in  Theorem~\ref{theorem.IcondnX}, it suffices to show that $\|X\|$ being bounded a.s.
and $I(\epsilon)<+\infty$ for some $\epsilon>0$  implies  \eqref{eq.IcondnXbound}. We note that
if there is an $\epsilon^\ast>0$ such that $I(\epsilon^\ast)<+\infty$, then there is an $\epsilon>0$ such that $S_1(\epsilon)<+\infty$. Hence it only remains to prove that  $\|X\|$ being bounded a.s.
and $S_1(\epsilon)<+\infty$ for some $\epsilon>0$  implies  \eqref{eq.IcondnXbound}.

To do this, we note that  $S_1(\epsilon)<+\infty$ for some $\epsilon>0$ implies $\int_n^{n+1}\|\sigma(s)\|^2_F\,ds\to 0$ as $n\to\infty$. In turn, this implies
\begin{equation}  \label{eq.cesarosigF}
\lim_{t\to\infty} \frac{1}{t}\int_0^t \|\sigma(s)\|^2_F\,ds=0.
\end{equation}
Since all the eigenvalues of $A$ have negative real part, there
exists a $d\times d$ positive definite matrix $M$ such that
\begin{equation}\label{def.Mmatrix}
A^TM+MA=-I_d.
\end{equation}
(see for example Horn and Johnson~\cite{HornJohnson:1991} or
Rugh~\cite{Rugh:1996}). Define $V(x)=x^TMx$ for all
$x\in\mathbb{R}^d$. Notice that
\[
\frac{\partial V}{\partial x_i}= [2Mx]_i=\sum_{k=1}^d 2M_{ik}x_k.
\]
Therefore we have
\[
\frac{\partial^2 V}{\partial x_i\partial x_j}(x)=2M_{ij}.
\]
Let $X_i(t)=\langle X(t),e_i\rangle$. Notice that the
cross--variation of $X_i$ and $X_j$ obeys
\[
d\langle X_i,X_j\rangle(t)=\sum_{k=1}^r
\sigma_{ik}(t)\sigma_{jk}(t)\,dt.
\]
Therefore, as $V$ is a $C^2$ function, by the multidimensional
version of It\^o's formula, we have
\[
dV(X(t))=\sum_{i=1}^d \frac{\partial V}{\partial x_i}(X(t)) dX_i(t)
+ \frac{1}{2}\sum_{i=1}^d \sum_{j=1}^d \frac{\partial^2 V}{\partial
x_i\partial x_j}(X(t)) d\langle X_i,X_j\rangle(t).
\]
Hence
\begin{multline*}
dV(X(t))=\langle 2MX(t),AX(t)\rangle\,dt  + \sum_{i=1}^d \sum_{j=1}^d  M_{ij} \sum_{k=1}^r \sigma_{ik}(t)\sigma_{jk}(t)\,dt \\
+  \langle 2MX(t),\sigma(t)\,dB(t)\rangle.
\end{multline*}
Next, we note that because $M=M^T$ and $A^TM+MA=-I_d$, we have
\begin{multline*}
\langle 2Mx,Ax\rangle=\langle (M+M^T)x,Ax\rangle =\langle
Mx,Ax\rangle +\langle Ax,M^Tx\rangle
\\=(Mx)^TAx+(Ax)^T M^Tx
=x^TM^TAx+x^TA^TMx=-x^T x.
\end{multline*}
Also, since $M$ is positive definite, there exists a $d\times d$
matrix $P$ such that $M=PP^T$, so we have
\begin{align*}
\sum_{i=1}^d \sum_{j=1}^d  M_{ij} \sum_{k=1}^r
\sigma_{ik}(t)\sigma_{jk}(t) &=
\sum_{i=1}^d  \sum_{k=1}^r  \left(\sum_{j=1}^d  M_{ij}\sigma_{jk}(t) \right)\sigma_{ik}(t)\\
&= \sum_{i=1}^d  \sum_{k=1}^r   [M \sigma(t)]_{ik}  \sigma^T_{ki}(t)
=
\sum_{i=1}^d   [M \sigma(t)\sigma(t)^T]_{ii}\\
&= \sum_{i=1}^d   [PP^T \sigma(t)\sigma(t)^T]_{ii} =
\text{tr}(PP^T\sigma(t)\sigma(t)^T)\\
&= \text{tr}(P^T\sigma(t)\sigma(t)^TP)=\|P^T\sigma(t)\|_F^2.
\end{align*}
where we have used the fact that $\|C\|_F^2=\text{tr}(CC^T)$ for any
matrix $C$ and that $\text{tr}(CD)=\text{tr}(DC)$ for square
matrices $C$ and $D$. Thus
\begin{equation} \label{eq.masterV1}
V(X(t))=V(\xi)-\int_0^t X(s)^TX(s)\,ds + \int_0^t \|P^T\sigma(s)\|_F^2\,ds +  K(t), \quad t\geq 0,
\end{equation}
where
\begin{equation}\label{def.K}
K(t)= \sum_{j=1}^r \int_0^t  \left\{\sum_{i=1}^d [2MX(s)]_i
\sigma_{ij}(s)\right\} \,dB_j(s), \quad t\geq 0.
\end{equation}
We consider the third term on the righthand side of
\eqref{eq.masterV1}. Since $\|P^T\sigma(s)\|_F\leq
\|P^T\|_F\|\sigma(s)\|_F$, from \eqref{eq.cesarosigF}, we have that
\begin{equation} \label{eq.Psigcesaro}
\lim_{t\to\infty} \frac{1}{t}\int_0^t \|P^T\sigma(s)\|_F^2\,ds=0.
\end{equation}
 As to $K$, the fourth term on the righthand side of
\eqref{eq.masterV1}, we see that $K$ is a local martingale with quadratic variation given by
\[
\langle K \rangle(t)= \sum_{j=1}^r \int_0^t  \left\{\sum_{i=1}^d
[2MX(s)]_i  \sigma_{ij}(s)\right\}^2 \,ds.
\]
Hence by the Cauchy--Schwarz inequality, we have
\begin{equation} \label{eq.qvKest}
\langle K \rangle(t)\leq  \sum_{j=1}^r \int_0^t  \sum_{i=1}^d
[2MX(s)]_i^2 \sum_{i=1}^d \sigma_{ij}^2(s) \,ds = 4\int_0^t
\|MX(s)\|_2^2 \|\sigma(s)\|^2_F \,ds.
\end{equation}
Since $t\mapsto \|X(t)\|$ is bounded a.s., we may use
\eqref{eq.cesarosigF} to get
\[
\lim_{t\to\infty} \frac{1}{t}\langle K \rangle(t)=0,
\quad\text{a.s.}
\]
Hence by the strong law of large numbers for martingales, we have
that $K(t)/t\to 0$ as $t\to\infty$ a.s. Since $t\mapsto \|X(t)\|$ is
bounded a.s. we have that $V(X(t))/t\to 0$ as $t\to\infty$ a.s.
Therefore, returning to \eqref{eq.masterV1}, we get
\begin{equation}\label{eq.xNxdivtto0}
\lim_{t\to\infty} \frac{1}{t}\int_0^t X(s)^TX(s)\,ds =0,
\quad\text{a.s.}
\end{equation}
Suppose now that there is an event $A_1$ with $\mathbb{P}[A_1]>0$
such that
\[
A_1=\{\omega:\liminf_{t\to\infty} \|X(t,\omega)\|>0\}.
\]
Since $t\mapsto \|X(t)\|$ is bounded, it follows that for each
$\omega\in A_1$, there is a positive and finite $\bar{x}(\omega)$
such that
\[
\liminf_{t\to\infty} \|X(t,\omega)\|_2=:\bar{x}(\omega).
\]
Therefore for $\omega\in A_1$ we have
\begin{equation*}
\liminf_{t\to\infty} \frac{1}{t} \int_0^t
X(s,\omega)^TX(s,\omega)\,ds
\geq \bar{x}(\omega)>0.
\end{equation*}
Therefore
\[
\liminf_{t\to\infty} \frac{1}{t}\int_0^t X(s)^TX(s)\,ds>0, \quad
\text{a.s. on $A_1$},
\]
which contradicts \eqref{eq.xNxdivtto0}, because
$\mathbb{P}[A_1]>0$. Therefore, it must be the case that
$\mathbb{P}[A_1]=0$, which implies that
$\mathbb{P}[\overline{A}_1]=1$, or $\liminf_{t\to\infty} \|X(t)\|=0$
a.s. as required.

\section{Proofs of Proposition~\ref{prop.meansquare} and Part (C) of Theorem~\ref{theorem.liniffsigmafadeS}}
We prove a simple lemma which will be of utility in the proof of each of these results. 
\begin{lemma}  \label{lemma.econvfto0}
Suppose that $f:[0,\infty)\to \mathbb{R}$ is a continuous function such that 
\[
\lim_{n\to\infty} \int_{nh}^{(n+1)h} f^2(s)\,ds =0.
\]
Then for any $\lambda>0$ we have
\[
\lim_{t\to\infty} \int_{0}^t e^{-2\lambda(t-s)}f^2(s)\,ds =0.
\]
\end{lemma}
\begin{proof}
For every $t>0$ there exists $n(t)\in \mathbb{N}$ such that $n(t)h\leq t < (n(t)+1)h$. Then
\begin{align*}
 \lefteqn{\int_0^t e^{-2\lambda(t-s)}f(s)^2\,ds}\\
 &=
 \sum_{j=1}^{n(t)}
  \int_{(j-1)h}^{jh} e^{-2\lambda(t-s)}f^2(s)\,ds + \int_{n(t)h}^t e^{-2\lambda(t-s)}f^2(s)\,ds
  \\
  &\leq
  \sum_{j=1}^{n(t)} e^{-2\lambda h(n(t)-j)} \int_{(j-1)h}^{jh} f^2(s)\,ds + \int_{n(t)h}^{(n(t)+1)h} f^2(s)\,ds.
\end{align*}
Therefore, as the last term has zero limit because $\int_{nh}^{(n+1)h} f^2(s)\,ds\to 0$ as $n\to\infty$, we have
\[
\limsup_{t\to\infty}
 \int_0^t e^{-2\lambda (t-s)}\|\sigma(s)\|^2_F\,ds
 \leq
 \limsup_{n\to\infty} \sum_{j=1}^{n} e^{-2 \lambda h(n-j)} \int_{(j-1)h}^{jh} f^2(s)\,ds.
\]
We see that the righthand side is
the discrete convolution of a summable and a null sequence. Hence the limit is zero, and the claim holds.
\end{proof}

\subsection{Proof of Proposition~\ref{prop.meansquare}} 
It is easy to see that (A) implies (B), that (B) implies (C), and that (C) implies (A). Hence (A)--(C) are equivalent. 
We prove now (C) implies (D). 
Given that $X(0)=\xi$ is independent of $B$, It\^o's isometry yields
\[
\mathbb{E}[\|X(t)\|^2_2]=\mathbb{E}[\|\Psi(t)\xi\|^2_2]+ \int_0^t \|\Psi(t-s)\sigma(s)\|^2_F\,ds, \quad t\geq 0.
\]
Since all the eigenvalues of $A$ have negative real parts, it follows that there exists $\lambda>0$ 
and $K_2>0$ such that $\|\Psi(t)\|_2\leq K_2e^{-\lambda t}$ for all $t\geq 0$. Since there exists a $c_1>0$ such that 
$\|C\|_F\leq c_1\|C\|_2$ for all $C\in \mathbb{R}^{d\times d}$, we have 
\[
\|\Psi(t)e^{\lambda t}\|_F\leq c_1\|\Psi(t)e^{\lambda t}\|_2\leq c_1 K_2, \quad t\geq 0,
\]
or $\|\Psi(t)\|_F\leq c_1K_2 e^{-\lambda t}$ for all $t\geq 0$. Hence using the submultiplicative property of the Frobenius norm, we have 
\begin{align*}
\mathbb{E}[\|X(t)\|^2_2]
&\leq \|\Psi(t)\|^2_2\mathbb{E}[\|\xi\|^2_2]+ \int_0^t \|\Psi(t-s)\|^2_F\|\sigma(s)\|^2_F\,ds\\
&\leq  K_2^2 e^{-2\lambda t}\mathbb{E}[\|\xi\|^2_2] + c_1^2 K_2^2\int_0^t  e^{-2\lambda(t-s)}\|\sigma(s)\|^2_F\,ds.
\end{align*}
By Lemma~\ref{lemma.econvfto0}, the second term on the righthand side tends to zero as $t\to\infty$ when $\sigma$ obeys \eqref{eq.sighto0}, which proves that statement (A) implies statement (D). 

To prove that statement (D) implies statement (C), which will suffice to complete the proof, we start by writing  
\[
\int_t^{t+1} \sigma(s)\,dB(s)=X(t+1)-X(t)-\int_t^{t+1} AX(s)\,ds, \quad t\geq 0.
\] 
Considering the expectation of  $\|\cdot\|^2$ on both sides, and using It\^o's isometry on the left hand side, we deduce the identity
\[
\int_t^{t+1} \|\sigma(s)\|^2_F\,ds = \mathbb{E}\left[\left\|X(t+1)-X(t)-\int_t^{t+1} AX(s)\,ds\right\|^2_2  \right].
\]
Since (D) holds by hypothesis, the righthand side converges to zero as $t\to\infty$, completing the proof. 

\subsection{Proof of Part (C) of Theorem~\ref{theorem.liniffsigmafadeS}} 
In part (C), $\sigma$ is not in $L^2([0,\infty);\mathbb{R}^{d\times r})$.
In this case, there exists a pair of integers $(i,j)\in \{1,\ldots,d\}\times\{1,\ldots,r\}$
such that $\sigma_{ij}\not\in L^2([0,\infty);\mathbb{R})$.
Note that $Y_i$ obeys
\[
dY_i(t)=-Y_i(t)\,dt + \sum_{j=1}^r \sigma_{ij}(t)\,dB_j(t), \quad
t\geq 0.
\]
Thus there exists a standard Brownian motion $\bar{B}_i$ such that
\[
dY_i(t)=-Y_i(t)\,dt + \sqrt{\sum_{l=1}^r \sigma_{il}^2(t)}\,d\bar{B}_i(t), \quad t\geq 0.
\]
Define
\begin{equation} \label{def.sigmai}
\sigma_i^2(t)=\sum_{l=1}^r \sigma_{il}^2(t), \quad t\geq 0.
 \end{equation}
 Then $\sigma_i\not\in L^2(0,\infty)$, and it is possible to define a number $T_i>0$ such that $\int_0^t e^{2s}\sigma_i^2(s)\,ds>e^e$
 for $t>T_i$ and so one can define a function $\Sigma_i:[T_i,\infty)\to [0,\infty)$ by
\begin{equation} \label{def.Sigmai}
\Sigma_i(t)=\left( \int_0^t e^{-2(t-s)}\sigma_i^2(s)\,ds
\right)^{1/2} \left(\log\log \int_0^t e^{2s}\sigma_i^2(s)\,ds
\right)^{1/2}, \quad t\geq T_i.
\end{equation}
Notice also that for $t\geq T_i$ we have
\begin{align*}
\Sigma_i(t)^2&\leq  \int_0^t e^{-2(t-s)}\sigma_i^2(s)\,ds \cdot \log\log \int_0^t e^{2s}\|\sigma(s)\|^2_F\,ds\\
&\leq \int_0^t e^{-2(t-s)}\|\sigma(s)\|^2_F\,ds \cdot \log\log e^{2t}\int_0^t \|\sigma(s)\|^2_F\,ds.
\end{align*}
The significance of the function $\Sigma_i$ defined in
\eqref{def.Sigmai} is that it characterises the largest possible
fluctuations of $Y_i$ when $\sigma_i\not\in L^2(0,\infty)$. To do this we apply the Law of the
iterated logarithm for martingales to $M(t):=\int_0^t e^s
\sigma_i(s)\,d\bar{B}_i(s)$. This holds because $\sigma_i\not \in
L^2([0,\infty);\mathbb{R}^{d\times r})$ implies that $\langle
M\rangle (t)=\int_0^t e^{2s} \sigma_i^2(s)\,ds\to\infty$ as
$t\to\infty$. We get
\begin{equation} \label{eq.YiasySig} \limsup_{t\to\infty}
\frac{\|Y_i(t)\|^2}{\Sigma_i^2(t)}=2, \quad\text{a.s.}
\end{equation}
Let $N=\{i=1,\ldots,d: \sigma_i\not\in L^2(0,\infty)\}$, and $F=\{1,\ldots,d\}\setminus N$.
Clearly, if $i\in F$, we have that $Y_i(t)\to 0$ as $t\to\infty$ a.s., so
\[
\lim_{t\to\infty} \sum_{i\in F} \|Y_i(t)\|^2 =0, \quad\text{a.s.}
\]
By \eqref{eq.YiasySig}, for every $i\in N$, there exist $T_i'(\omega)>T_i$ such that
$\|Y_i(t,\omega)\|^2\leq 4 \Sigma_i^2(t)$ for $t\geq T_i'(\omega)$. Define $T^\ast(\omega)=\max_{i\in N} T_i(\omega)$. Then  for $t\geq T^\ast(\omega)$
we have
\[
\|Y_i(t,\omega)\|^2\leq 4 \Sigma_i^2(t) \leq 4 \int_0^t e^{-2(t-s)}\|\sigma(s)\|^2_F\,ds \cdot \log\log \left(e^{2t}\int_0^t \|\sigma(s)\|^2_F\,ds\right).
\]
Therefore for $t\geq T^\ast(\omega)$ we get
\[
\sum_{i\in N} \|Y_i(t,\omega)\|^2 \leq 4d \int_0^t e^{-2(t-s)}\|\sigma(s)\|^2_F\,ds \cdot \log\log \left(e^{2t}\int_0^t \|\sigma(s)\|^2_F\,ds\right).
\]
Hence there exists $T'(\omega)>0$ such that for all $ t\geq T'(\omega)$ we have
\[
\|Y(t,\omega)\|^2\leq 1+4d \int_0^t e^{-2(t-s)}\|\sigma(s)\|^2_F\,ds \cdot \log\log \left(e^{2t}\int_0^t \|\sigma(s)\|^2_F\,ds\right). 
\]
Now, because \eqref{eq.sighto0} holds, we have $\int_0^t \|\sigma(s)\|^2_F\,ds/t\to 0$ as $t\to\infty$. Therefore
\[
\limsup_{t\to\infty} \frac{1}{\log t} \log\log \left(e^{2t}\int_0^t \|\sigma(s)\|^2_F\,ds\right) \leq 1.
\]
Hence there is $T''(\omega)>0$ such that for all $t\geq T''(\omega)$ we have
\[
\|Y(t,\omega)\|^2\leq 1+8d \int_0^t e^{-2(t-s)}\|\sigma(s)\|^2_F\,ds \cdot \log t, \quad t\geq T''(\omega).
\]
Hence
\[ 
\limsup_{t\to\infty} \frac{\|Y(t)\|^2_2}{\log t}\leq 8d \limsup_{t\to\infty} \int_0^t e^{-2(t-s)}\|\sigma(s)\|^2_F\,ds, \quad\text{a.s.}
\]
By Lemma~\ref{lemma.econvfto0}, we have that 
\begin{equation*} 
\lim_{t\to\infty} \int_0^t e^{-2(t-s)}\|\sigma(s)\|^2_F\,ds=0,
\end{equation*}
which ensures that
\begin{equation} \label{eq.masterlimsupY2}
\lim_{t\to\infty} \frac{\|Y(t)\|^2_2}{\log t} = 0, \quad\text{a.s.}
\end{equation}

Using the proof of part (A) of Theorem~\ref{theorem.XYasy}, we have from \eqref{eq.zvarp} and \eqref{def.g}
that $z(t):=X(t)-Y(t)$ for $t\geq 0$ obeys
\begin{equation*} 
z(t)=\Psi(t)\xi+\int_0^t \Psi(t-s)(I_d+A)Y(s)\,ds, \quad
t\geq 0.
\end{equation*}
Since $\Psi$ obeys the estimate \eqref{eq.Psiexpest}, we have for $t\geq 0$
\begin{align*}
\|z(t)\|_2
&\leq \|\Psi(t)\|_2\|\xi\|_2+\int_0^t \|\Psi(t-s)\|_2\|I_d+A\|_2\|Y(s)\|_2\,ds\\
&\leq Ke^{-\lambda t}\|\xi\|_2+K\|I_d+A\|_2 \int_0^t e^{-\lambda (t-s)} \|Y(s)\|_2\,ds.
\end{align*}
Therefore we have
\[
\|X(t)\|_2\leq Ke^{-\lambda t}\|\xi\|_2+\|Y(t)\|_2+K\|I_d+A\|_2 \int_0^t e^{-\lambda (t-s)} \|Y(s)\|_2\,ds, \quad t\geq 0.
\]
Since $Y$ obeys \eqref{eq.masterlimsupY2}, it follows that
\begin{equation} \label{eq.Xlogt0}
\limsup_{t\to\infty} \frac{\|X(t)\|_2}{\sqrt{\log t}}=0, \quad\text{a.s.}
\end{equation}

Our strategy now is to return to the identity \eqref{eq.masterV1}, and estimate the asymptotic behaviour of each of the terms. We start with the term on the lefthand side. Since all the eigenvalues of $A$ have negative real parts, there exists a positive definite matrix $M$ which satisfies  \eqref{def.Mmatrix}. Then $V(x)=\langle x, Mx\rangle$ obeys
\[
\frac{V(x)}{\|x\|^2_2} = \langle \frac{x}{\|x\|_2}, M\frac{x}{\|x\|_2}\rangle \leq \sup_{\|u\|_2=1} \langle u,Mu\rangle=:\mu_1>0
\]
for $x\neq 0$. Hence $0\leq V(x)\leq \mu_1\|x\|^2_2$ for all $x\in\mathbb{R}^d$. Therefore by \eqref{eq.Xlogt0} we have
\begin{equation}\label{eq.VXlogt0}
\lim_{t\to\infty} \frac{V(X(t))}{\log t}=0, \quad\text{a.s.}
\end{equation}
The first term on the righthand side of \eqref{eq.masterV1} is constant. We wish to prove that the second term on the righthand side of \eqref{eq.masterV1} obeys
\eqref{eq.xNxdivtto0}, i.e.,
\[
\lim_{t\to\infty} \frac{1}{t}\int_0^t \|X(s)\|^2_2\,ds =0, \quad\text{a.s.}
\]
We note that this limit automatically implies that $\liminf_{t\to\infty} \|X(t,\xi)\|_2=0$ a.s.

The asymptotic behaviour of the third term on the righthand side of \eqref{eq.masterV1} is easily determined: since \eqref{eq.sighto0} holds, the limit \eqref{eq.Psigcesaro} follows. It remains to estimate the asymptotic behaviour of the fourth term on the righthand side of \eqref{eq.masterV1}, which is a local martingale with quadratic variation bounded by \eqref{eq.qvKest}, i.e.,
\begin{equation*}
\langle K \rangle(t)\leq  4\|M\|^2_2\int_0^t \|X(s)\|_2^2 \|\sigma(s)\|^2_F \,ds, \quad t\geq 0.
\end{equation*}
By \eqref{eq.Xlogt0}, it follows for every $\epsilon>0$ and $\omega$ in an a.s. event $\Omega^\ast$ that there is a $T_1(\epsilon,\omega)$ such that
\[
\|X(t,\omega)\|^2_2 < \epsilon\log t, \quad t\geq T_1(\epsilon,\omega).
\]
By \eqref{eq.sighto0}, we have that there exists $T_2(\epsilon)>0$ such that $\int_0^t \|\sigma(s)\|^2_F\,ds < \epsilon t$ for $t\geq T_2(\epsilon)$.
Define $T_3(\epsilon,\omega)=\max(T_1(\epsilon,\omega),T_2(\epsilon))$. Then for $t\geq T_3(\epsilon,\omega)$ we have
\begin{align*}
\langle K \rangle(t)&\leq  
D(\epsilon,\omega)+4\|M\|^2_2\int_{T_3(\epsilon,\omega)}^t  \|X(s,\omega)\|_2^2 \|\sigma(s)\|^2_F \,ds\\
&\leq  
D(\epsilon,\omega)+4\|M\|^2_2\epsilon \log t \int_{T_3(\epsilon,\omega)}^t  \|\sigma(s)\|^2_F \,ds\\
&\leq  
D(\epsilon,\omega)+4\|M\|^2_2\epsilon^2 t \log t,
\end{align*}
where we have defined
\[
D(\epsilon,\omega):=4\|M\|^2_2\int_0^{T_3(\epsilon,\omega)} \|X(s,\omega)\|_2^2 \|\sigma(s)\|^2_F \,ds.
\]
Hence we have that
\begin{equation} \label{eq.qvestKexp}
\lim_{t\to\infty} \frac{\langle K\rangle (t)}{t\log t}=0, \quad\text{a.s.}
\end{equation}
Let $A_1:=\{\omega: \lim_{t\to\infty} \langle K\rangle(t,\omega) \text{ is finite}\}$, and
$A_2:=\{\omega: \lim_{t\to\infty} \langle K\rangle(t,\omega) =+\infty\}$. Then $K$ converges a.s. on $A_1$ and we have
\[
\lim_{t\to\infty} \frac{1}{t}K(t)=0, \quad\text{a.s. on $A_1$.}
\]
On $A_2$, the Law of the iterated logarithm for martingales holds, namely
\[
\limsup_{t\to\infty} \frac{|K(t)|}{\sqrt{2\langle K\rangle(t)\log\log \langle K\rangle (t)}}=1, \quad \text{a.s. on $A_2$}.
\]
By \eqref{eq.qvestKexp} we have
\begin{equation}\label{eq.log2qvK}
\limsup_{t\to\infty} \frac{\log\log \langle K\rangle(t)}{\log_2 t}\leq 1, \quad \text{a.s. on $A_2$}.
\end{equation}
Therefore, we have
\[
\limsup_{t\to\infty} \frac{|K(t)|}{t}\leq  
\limsup_{t\to\infty} \sqrt{\frac{2\langle K\rangle(t)\log\log \langle K\rangle (t)}{t^2}}, \quad \text{a.s. on $A_2$}
\]
Now, we rewrite the quotient in the limit according to
\[
\frac{2\langle K\rangle(t)\log\log \langle K\rangle (t)}{t^2}
=2\frac{\langle K\rangle(t)}{t\log t}\cdot \frac{\log t \cdot \log\log t}{t} \cdot \frac{\log\log \langle K\rangle (t)}{\log\log t},
\]
and so from \eqref{eq.qvestKexp} and \eqref{eq.log2qvK}, we have that
\[
\lim_{t\to\infty} \frac{K(t)}{t}=0, \quad \text{a.s. on $A_2$}
\]
Since $A_1\cup A_2$ is an a.s. event, it follows that $K(t)/t\to 0$ as $t\to\infty$ a.s. Using this limit, \eqref{eq.Psigcesaro}, and \eqref{eq.VXlogt0}
in \eqref{eq.masterV1}, we arrive at the desired limit \eqref{eq.xNxdivtto0}.

\section{Proof of Theorem~\ref{theorem.stochfloq}}
Under \eqref{eq.floqmultlt1}, 
By \cite{Chicone:1999}[Theorem 2.48], we have that there exists a continuously differentiable function such that $P(t)\in \mathbb{C}^{d\times d}$, $P(t)$ is invertible and $P$ is $T$--periodic, and a matrix $L\in \mathbb{C}^{d\times d}$ all of whose eigenvalues have negative real parts such that 
\[
\Psi(t)=P(t)e^{Lt}.
\]   
Notice also that $P^{-1}$ is continuously differentiable and $T$--periodic. 
Since all the eigenvalues of $L$ have negative real parts, there exists a Hermitian and positive definite matrix $Q\in \mathbb{C}^{d\times d}$
such that 
\[
QL+L^\ast Q=-I_d.
\]
Also, as $P$ is periodic and continuous, and $P^{-1}$ is periodic and continuous, 
we have the estimate $\|P(t)\|\leq p_\ast$, $\|P(t)^{-1}\|\leq p_\ast$ for some $p_\ast>0$. Also, as 
all eigenvalues of $L$ have negative real parts, we have the estimate
\[
\|\Psi(t)\|\leq p_\ast e^{-\lambda t}, \quad \|e^{Lt}\|\leq ce^{-\lambda t}.
\]

Define $z(t)=X(t)-Y(t)$ for $t\geq 0$. Then with $g(t)=(I_d+A(t))Y(t)$, we have 
\[
z'(t)=A(t)Z(t)+g(t), \quad t\geq 0; \quad z(0)=\xi.
\]
Hence for $t\geq 0$ we have the variation of constants formula
\[
z(t)=\Psi(t)\xi+\int_{0}^t \Psi(t)\Psi(s)^{-1}g(s)\,ds
=\Psi(t)\xi+\int_{0}^t P(t)e^{L(t-s)}P(s)^{-1} g(s)\,ds.
\]
Therefore for $t\geq 0$ we have 
\[
\|z(t)\|\leq  p_\ast e^{-\lambda t}\|\xi\| +  p_\ast^2 c \max_{s\in [0,T]} \|I+A(s)\| \cdot \int_{0}^t e^{-\lambda(t-s)}  \|Y(s)\|\,ds.
\]
This leads to the estimate
\begin{equation} \label{eq.Xtestperaff}
\|X(t)\|\leq p_\ast e^{-\lambda t}\|\xi\| + \|Y(t)\| + p_\ast^2 c \max_{s\in [0,T]} \|I+A(s)\| \cdot \int_{0}^t e^{-\lambda(t-s)}  \|Y(s)\|\,ds.
\end{equation}
We see automatically that when $Y(t)\to 0$ as $t\to\infty$ a.s., then $X(t)\to 0$ as $t\to\infty$ a.s.; this proves part (A), because 
$S_h'(\epsilon)<+\infty$ implies $\lim_{t\to\infty} Y(t)=0$ a.s. 

In the case that $S_h'(\epsilon)<+\infty$ for all $\epsilon>\epsilon'$ and $S_h'(\epsilon)=+\infty$ 
for all $\epsilon<\epsilon'$, we have that $\limsup_{t\to\infty} \|Y(t)\|\leq c_2$ a.s. for some deterministic $c_2>0$. Therefore, from 
\eqref{eq.Xtestperaff}, we see that $\limsup_{t\to\infty} \|X(t)\|\leq c_4$ a.s., where $c_4$ is
\[
c_4=c_2+p_\ast^2 c \max_{s\in [0,T]} \|I+A(s)\| \frac{1}{\lambda} c_2,
\] 
which yields the desired upper bound in part (B).

To prove part (C), we start by noticing that $S_h'(\epsilon)=+\infty$ for all $\epsilon>0$ implies $\limsup_{t\to\infty} \|Y(t)\|=+\infty$ a.s.
Observing that the identity 
\[
Y(t)= 
X(t)-\xi e^{-t} - \int_{0}^t e^{-(t-s)}(I_d+A(s))X(s)\,ds, \quad t\geq 0,
\]
holds, we see that if there is an event of positive probability for which the limit superior $\limsup_{t\to\infty} \|X(t)\|$ is finite, then $\limsup_{t\to\infty} \|Y(t)\|<+\infty$
on this event, which results in a contradiction.  

The proof of the lower bound in part (B) is similar. Since $S_h'(\epsilon)<+\infty$ for all $\epsilon>\epsilon'$ and $S_h'(\epsilon)=+\infty$ for 
all $\epsilon<\epsilon'$, it follows that there exists a deterministic $c_1>0$ such that $\limsup_{t\to\infty} \|Y(t)\|\geq c_1$ a.s.
Suppose that there is an event of positive probability such that $\limsup_{t\to\infty}\|X(t)\|_2=:c(\omega)<c_1/(1+\max_{t\in [0,T]} \|I_d+A(t)\|_2)=:c_3$. 
Then 
\[
c_1\leq \limsup_{t\to\infty} \|Y(t)\|\leq c(\omega) \|X(t)\|+ \sup_{t\in [0,T]} \|I_d+A(t)\|_2 \cdot c(\omega), 
\]
so $c_1/(1+\max_{t\in [0,T]} \|I_d+A(t)\|_2)>c(\omega)\geq c_1/(1+\max_{t\in [0,T]} \|I_d+A(t)\|_2)$, a conrtadiction.
Therefore we have that $\limsup_{t\to\infty} \|X(t)\|\geq c_3$ a.s. 

We now prove the ergodic result in part (B), from which $\liminf_{t\to\infty} \|X(t)\|=0$ a.s. follows easily.
To do this, we define for $x\in\mathbb{R}^d$ the function $V(t,x)=x^T (P(t)^{-1})^\ast QP(t)^{-1}x$. Note that $V(\cdot,x)$ is $T$--periodic and real--valued, 
because $M(t):=(P(t)^{-1})^\ast QP(t)^{-1}$ is Hermitian. We may now write $V(t,x)=x^T M(t)x$. This function $V$ 
was used in Giesl and Hafstein~\cite[Theorem 6]{GieslHaf:2012} as a strict Lyapunov function in proving that the zero solution of 
the unperturbed differential equation $x'(t)=A(t)x(t)$ is asymptotically stable. 

We start by obtaining a $t$--uniform upper bound on $V$. Define $M_1(t)=M(t)+M(t)^t$. 
Suppressing $t$ dependence for a moment, we notice that $M_1=M+M^T$ is symmetric. Also, if we define the matrices $G,H\in \mathbb{R}^{d\times d}$ so that $M=G+iH$, then $M^T=G^T+iH^T$ and $M^\ast=G^T-iH^T$. Therefore as $M=M^\ast$, we have $G=G^T$ and $H=-H^T$. Hence $M_1=M+M^T=(G+G^T)+i(H+H^T)=2G$ is a real--valued symmetric matrix. 

For $x\neq 0$, we now have 
\begin{align*}
\frac{V(t,x)}{\|x\|^2_2}&=\frac{x^T M(t)x}{\|x\|^2}
\leq \sup_{\|u\|_2=1;u\in \mathbb{R}^d} u^T M(t)u\\
&=\sup_{\|u\|_2=1;u\in \mathbb{R}^d} \frac{1}{2} u^T M_1(t)u\leq \frac{1}{2}\|M_1(t)\|_2.
\end{align*}
Since $t\mapsto P(t)^{-1}$ is continuous and $T$--periodic, it follows that $t\mapsto M_1(t)$ is continous, real--valued and $T$--periodic. Therefore, 
there exists $c_6\in (0,\infty)$ defined by $c_6:=\max_{t\in [0,T]}\|M_1(t)\|_2/2$ such that 
\begin{equation} \label{eq.Vuplow}
V(t,x)\leq c_6\|x\|^2_2, \quad \text{ for all $(t,x)\in [0,\infty)\times \mathbb{R}^d$}.
\end{equation}
Next, we notice that 
\[
\dot{P}^{-1}(t) = LP^{-1}(t)-P^{-1}(t)A(t).
\]
Therefore 
\begin{align*}
M'(t)&=(\dot{P}(t)^{-1})^\ast QP(t)^{-1} + (P(t)^{-1})^\ast Q\dot{P}(t)^{-1}\\
&=(LP^{-1}(t)-P^{-1}(t)A(t))^\ast QP(t)^{-1} + (P(t)^{-1})^\ast Q ( LP^{-1}(t)-P^{-1}(t)A(t)).
\end{align*}
Hence
\begin{multline*}
M'(t)=
P^{-1}(t)^\ast L^\ast  QP(t)^{-1} -A(t)^\ast P^{-1}(t)^\ast  QP(t)^{-1} + (P(t)^{-1})^\ast Q LP^{-1}(t)
\\- (P(t)^{-1})^\ast QP^{-1}(t)A(t).
\end{multline*}
Using the fact that $QL+L^\ast Q=-I_d$, and the definition of $M(t)$ we get
\[ 
M'(t)=-P^{-1}(t)^\ast P(t)^{-1} -A(t)^\ast M(t) - M(t)A(t).
\]
Hence 
\[
\frac{\partial V}{\partial t}(t,x)=x^T M'(t)x=-x^T P^{-1}(t)^\ast P(t)^{-1}x -x^TA(t)^T M(t)x - x^TM(t)A(t)x.
\]
Next, we notice that
\[
\frac{\partial V}{\partial x_i}(t,x)= [(M(t)+M(t)^T)x]_i=\sum_{k=1}^d (M_{ik}(t)+M(t)^T_{ik})x_k.
\]
Therefore  we have
\[
\frac{\partial^2 V}{\partial x_i\partial x_j}(t,x)=[M_1(t)]_{ij}.
\]
Let $X_i(t)=\langle X(t),e_i\rangle$. Notice that the
cross--variation of $X_i$ and $X_j$ obeys
\[
d\langle X_i,X_j\rangle(t)=\sum_{k=1}^r
\sigma_{ik}(t)\sigma_{jk}(t)\,dt.
\]
Therefore, as $V$ is a $C^{1,2}$ function, by the multidimensional
version of It\^o's formula, we have
\begin{multline*}
dV(t,X(t))\\
=\Biggl(-X(t)^T P^{-1}(t)^\ast P(t)^{-1}X(t) -X(t)^TA(t)^T M(t)X(t) - X(t)^TM(t)A(t)X(t) \\ +  (M(t)+M(t)^T)X(t)^T A(t)X(t)
+ \frac{1}{2}\sum_{i=1}^d \sum_{j=1}^d  [M_1(t)]_{ij} \sum_{k=1}^r \sigma_{ik}(t)\sigma_{jk}(t)  \Biggr)\,dt \\
+  \langle (M(t)+M(t)^T)X(t),\sigma(t)\,dB(t)\rangle.
\end{multline*}
Since $M_1$ is a real--valued symmetric matrix, we may define the real--valued and deterministic function $J:[0,\infty)\to\mathbb{R}$ by 
\begin{equation} \label{def.Jaffineper}
J(t):=
\frac{1}{2}\sum_{i=1}^d \sum_{j=1}^d  (M_{ij}(t)+M_{ji}(t)) \sum_{k=1}^r \sigma_{ik}(t)\sigma_{jk}(t),
\end{equation}
and the real--valued continuous local martingale $K$ by
\begin{equation} \label{def.Kaffine}
K(t)=\int_0^t \langle M_1(s)X(s),\sigma(s)\,dB(s)\rangle=  \sum_{j=1}^r \int_0^t \left\{\sum_{i=1}^d [M_1(s)X(s)]_i \sigma(s)_{ij}\right\}\,dB_j(s),
\end{equation}
and observe that 
\begin{multline} \label{eq.Vtmaster}
V(t,X(t))
=V(0,\xi) - \int_0^t X(s)^T (P(s)^{-1})^\ast P(s)^{-1}X(s)\,ds \\
+ \int_0^t J(s)\,ds + K(t), \quad t\geq 0.
\end{multline}

We now attempt to estimate each of the terms in \eqref{eq.Vtmaster}. We start with $J(t)$, observing that it can be written as  
\begin{align*}
J(t)&=\frac{1}{2}\sum_{i=1}^d \sum_{j=1}^d  (M(t)^T+M(t))_{ji} (\sigma(t)\sigma(t)^T)_{ij} \\
&=\frac{1}{2}  \sum_{j=1}^d \sum_{i=1}^d (M(t)^T+M(t))_{ji} (\sigma(t)\sigma(t)^T)_{ij} \\
&=\frac{1}{2}  \text{tr}(M_1(t)\sigma(t)\sigma(t)^T).
\end{align*}
Since $t\mapsto P^{-1}(t)$ is $T$--periodic and continuous, it follows that $t\mapsto M_1(t)$ is continuous and $T$--periodic. Therefore, using the fact 
that the Frobenius norm is subadditive and submultiplicative, $\|D^T\|_F=\|D\|_F$ for every $d\times r$ matrix $D$, 
and that $\text{tr}(C)^2\leq d\|C\|^2_F$ for every $d\times d$ matrix $C$, we have that 
\begin{align*}
|J(t)|&= \frac{1}{2}  \left|\text{tr}(M_1(t)\sigma(t)\sigma(t)^T)\right| \leq \frac{1}{2}\sqrt{d}\|M_1(t)\sigma(t)\sigma(t)^T\|_F\\
&\leq \frac{1}{2}\sqrt{d}\|M_1(t)\|_F \|\sigma(t)\|_F \|\sigma(t)^T\|_F\\
&\leq  \frac{1}{2}\sqrt{d}\max_{t\in [0,T]}\|M_1(t)\|_F \cdot \|\sigma(t)\|^2_F.
\end{align*}
Now, as $S_h'(\epsilon)<+\infty$ for $\epsilon>\epsilon'$, it follows that $\int_{nh}^{(n+1)h} \|\sigma(s)\|^2_F\,ds\to 0$ as $n\to\infty$. Hence 
$\lim_{t\to\infty} t^{-1}\int_0^t \|\sigma(s)\|^2_F\,ds=0$. Therefore, it follows that 
\begin{equation} \label{eq.cesaroLto0}
\lim_{t\to\infty} \frac{1}{t}\int_0^t J(s)\,ds = 0.
\end{equation}
Next we deal with the local martingale $K$ defined in \eqref{def.Kaffine}. We start by observing that it has quadratic variation 
given by 
\begin{equation*} 
\langle K\rangle (t)
=\sum_{j=1}^r \int_0^t \left\{\sum_{i=1}^d [M_1(s)X(s)]_i \sigma(s)_{ij}\right\}^2\,ds.
\end{equation*}
Therefore applying the Cauchy--Schwartz inequality, we have
\begin{align*} 
\langle K\rangle (t)
&\leq 
\sum_{j=1}^r \int_0^t \sum_{i=1}^d [M_1(s)X(s)]_i^2 \sum_{i=1}^d \sigma(s)_{ij}^2\,ds\\
&=
\int_0^t \|M_1(s)X(s)\|_2^2 \|\sigma(s)\|^2_F\,ds\\
&\leq 
\int_0^t \|M_1(s)\|^2_2\|X(s)\|_2^2 \|\sigma(s)\|^2_F\,ds.
\end{align*}
Now, as $M_1$ is continuous and $T$--periodic, it follows that 
\begin{equation} \label{eq.Kqvaffper}
\langle K\rangle(t)\leq \max_{t\in [0,T]} \|M_1(s)\|^2_2 \sup_{0\leq s\leq t} \|X(s)\|^2_2  \cdot \int_0^t  \|\sigma(s)\|^2_F\,ds, \quad t\geq 0.
\end{equation}
Therefore, as $t\mapsto \|X(t)\|$ is a.s. bounded, and $\int_0^t \|\sigma(s)\|_F^2\,ds/t\to 0$ as $t\to\infty$, we have 
\[
\lim_{t\to\infty}  \frac{\langle K\rangle(t)}{t}=0,\quad \text{a.s.}
\]
In the case that $\langle K\rangle(t)$ tends to a finite limit as $t\to\infty$, we have that $K(t)$ tends to a finite limit, and therefore 
that $\lim_{t\to\infty} K(t)/t=0$. If on the other hand $\langle K\rangle(t)\to \infty$ as $t\to\infty$, by the strong law of large numbers for martingales 
we have that $\lim_{t\to\infty} K(t)/\langle K\rangle(t)=0$. Therefore, in this case it follows that 
\[
\limsup_{t\to\infty} \frac{|K(t)|}{t}= \limsup_{t\to\infty}\frac{|K(t)|}{\langle K\rangle(t)}\cdot \frac{\langle K\rangle(t)}{t}=0.
\]
Therefore we have that 
\begin{equation} \label{eq.Ktto0affper}
\lim_{t\to\infty} \frac{1}{t}K(t)=0, \quad\text{a.s.}
\end{equation}
By \eqref{eq.Vuplow} and the fact that $X$ is bounded a.s. we have that 
\begin{equation} \label{eq.Vtt0affper}
\lim_{t\to\infty} \frac{1}{t}V(t,X(t))=0, \quad\text{a.s.}
\end{equation}
Therefore, inserting the estimates \eqref{eq.Vtt0affper}, \eqref{eq.Ktto0affper} and \eqref{eq.cesaroLto0} into \eqref{eq.Vtmaster}, we get 
\begin{equation} \label{eq.cesaroPx0}
\lim_{t\to\infty} \frac{1}{t}\int_0^t X(s)^T (P(s)^{-1})^\ast P(s)^{-1}X(s)\,ds =0, \quad\text{a.s.}
\end{equation}
For any $F\in \mathbb{C}^{d\times d}$, we have that $D=F^\ast F$ is Hermitian. Moreover, because $z^\ast D z=(Fz)^\ast Fz\geq 0$ for all $z\in \mathbb{C}^d$, it follows not only that $x^TDx$ is real--valued for every $x\in\mathbb{R}^d$, but also that $x^TDx\geq 0$ for all $x\in \mathbb{R}^d$ with equality only if $Fx=0$. 
Specialising to the case that $F=P(t)^{-1}$, we see that we have $x^T(P(t)^{-1})^\ast P(t)^{-1}x>0$ for all $x\neq 0$.
In fact, we have that 
\begin{align*}
\frac{x^T(P(t)^{-1})^\ast P(t)^{-1}x}{\|x\|^2_2}
&\geq \inf_{\|u\|_2=1;u\in\mathbb{R}^d} u^T(P(t)^{-1})^\ast P(t)^{-1}u\\
&\geq \inf_{\|u\|_2=1;u\in\mathbb{C}^d} (P(t)^{-1} u)^\ast P(t)^{-1}u=:\lambda(t)>0.
\end{align*}
Clearly, $\lambda$ is $T$--periodic and $\lambda(t)$ is the minimal eigenvalue of $(P(t)^{-1})^\ast P(t)^{-1}$. Since the matrix--valued function $(P(t)^{-1})^\ast P(t)^{-1}$ is continuous, $t\mapsto \lambda(t)$ is continuous and attains its bounds on the compact interval $[0,T]$. Therefore  
for all $x\in\mathbb{R}^d$ and $t\geq 0$, we have that there exists $c_7>0$ such that
\begin{equation} \label{eq.xPPxlowerbdd}
x^T (P(t)^{-1})^\ast P(t)^{-1}x \geq \min_{s\in [0,T]} \lambda(s) \cdot \|x\|^2_2=: c_7\|x\|_2^2.
\end{equation}
Therefore, applying this estimate in \eqref{eq.cesaroPx0}, we obtain 
\begin{equation}\label{eq.X2aveperaff}
\lim_{t\to\infty}\frac{1}{t}\int_0^t \|X(s)\|^2_2\,ds=0, \quad \text{a.s.} 
\end{equation}
from which we readily deduce $\liminf_{t\to\infty} \|X(t)\|_2=0$ a.s. 

\section{Proof of Theorem~\ref{theorem.stochfloqsmallnoise}}
In the case when $\sigma$ obeys \eqref{eq.sighto0}, we have already shown that $Y$ obeys \eqref{eq.masterlimsupY2}. 
Now, from \eqref{eq.Xtestperaff}, it follows that $X$ obeys the limit \eqref{eq.Xlogt0}. 
Due to \eqref{eq.Vuplow} and \eqref{eq.Xlogt0}, we have that \eqref{eq.Vtt0affper} holds. By \eqref{eq.sighto0}, $J$ defined by \eqref{def.Jaffineper}  obeys \eqref{eq.cesaroLto0}. Next, the local martingale $K$ defined by \eqref{def.Kaffine} has quadratic variation bounded by \eqref{eq.Kqvaffper}. 
Therefore from  \eqref{eq.Kqvaffper} we have  
\begin{equation} \label{eq.Kqvaffpertlogt}
\frac{\langle K\rangle(t)}{t \log t}\leq \max_{t\in [0,T]} \|M_1(s)\|^2_2 \frac{\sup_{0\leq s\leq t} \|X(s)\|^2_2}{\log t}  \cdot \frac{1}{t}\int_0^t  \|\sigma(s)\|^2_F\,ds, \quad t\geq 1,
\end{equation}
Since $\sigma$ obeys \eqref{eq.sighto0}, we have that $\int_0^t \|\sigma(s)\|^2_F\,ds/t\to 0$ as $t\to\infty$. Combining this estimate with 
\eqref{eq.Xlogt0} and \eqref{eq.Kqvaffpertlogt} we arrive at 
\[
\lim_{t\to\infty} \frac{\langle K\rangle(t)}{t\log t}=0, \quad \text{a.s.}
\]
Moreover, this implies 
\[
\limsup_{t\to\infty} \frac{\log\log \langle K\rangle(t)}{\log\log t}\leq 1, \quad\text{a.s.}
\] 
On the event on which $\langle K\rangle(t)$ tends to a finite limit as $t\to\infty$, it follows that $K$ tends to a finite limit a.s., and so we have 
that $K(t)/t\to 0$ as $t\to\infty$ a.s. on this event. On the other hand, consider the event on which $\langle K\rangle(t)\to\infty$ as 
$t\to\infty$. Then by the law of the iterated logarithm for martingales we have 
\begin{multline*}
\limsup_{t\to\infty} \frac{|K(t)|}{t}
\\=
\limsup_{t\to\infty} \frac{|K(t)|}{\sqrt{2\langle K\rangle(t)\log\log \langle K\rangle(t)}} 
\cdot \sqrt{\frac{2\langle K\rangle(t)}{t\log t}\frac{\log\log \langle K\rangle(t)}{\log\log t}  \cdot \frac{\log\log t\cdot \log t}{t}}
=0
\end{multline*}
a.s. on the event for which $\langle K\rangle(t)\to\infty$ as $t\to\infty$. Hence it follows that $K(t)/t\to 0$ as $t\to\infty$ a.s.

The representation \eqref{eq.Vtmaster} for $V(t,X(t))$ remains valid. Using the estimates \eqref{eq.Vtt0affper}, \eqref{eq.cesaroLto0}, and the fact that 
$K(t)/t\to 0$ as $t\to\infty$ a.s., we have that \eqref{eq.cesaroPx0} is true. Since the estimate \eqref{eq.xPPxlowerbdd} is still valid, this together 
with \eqref{eq.cesaroPx0} implies \eqref{eq.X2aveperaff}, as required. The conclusion that $\liminf_{t\to\infty} \|X(t)\|=0$ a.s. follows as before, completing 
the proof.

\end{document}